\documentclass[11pt,reqno,tbtags]{amsart}

\usepackage{amsthm,amssymb,amsfonts,amsmath}
\usepackage{amscd} 
\usepackage{mathrsfs}
\usepackage[numbers, sort&compress]{natbib} 
\usepackage{subcaption, graphicx}
\usepackage{vmargin}
\usepackage{setspace}
\usepackage{paralist}
\usepackage[pagebackref=true]{hyperref}
\usepackage{color}
\usepackage[normalem]{ulem}
\usepackage{stmaryrd} 
\usepackage[refpage,noprefix,intoc]{nomencl} 
\usepackage{tcolorbox}
\usepackage{bbm,mathtools,tikz}

\providecommand{\R}{}
\providecommand{\Z}{}
\providecommand{\N}{}

\providecommand{\U}{}

\renewcommand{\R}{\mathbb{R}}
\renewcommand{\Z}{\mathbb{Z}}
\renewcommand{\N}{{\mathbb N}}

\renewcommand{\P}{\mathbb{P}}

\renewcommand{\U}{\mathbb{U}}

\def \W {\mathsf{W}}

\newcommand{\E}[1]{{\mathbb{E}}\left[#1\right]}

\newcommand{\V}[1]{{\mathbf{Var}}\left\{#1\right\}}

\newcommand{\p}[1]{{\mathbb{P}}\left(#1\right)}

\newcommand{\I}[1]{{\mathbbm 1}_{\{#1\}}}
\newcommand{\un}{\mathbbm{1}}

\newcommand\cC{{\mathcal C}}

\newcommand\cF{\mathcal F}

\newcommand\cP{\mathcal P}

\newcommand\cS{{\mathcal S}}
\newcommand\cT{{\mathcal T}}

\usepackage{crossreftools,booktabs}

\makeatletter
\newcommand{\optionaldesc}[2]{%
  \phantomsection
  #1\protected@edef\@currentlabel{#1}\label{#2}%
}
\makeatother

\newcommand{\rA}{\mathrm{A}} 
\newcommand{\rB}{\mathrm{B}}

\newcommand{\rL}{\mathrm{L}} 
\newcommand{\rM}{\mathrm{M}}

\newcommand{\rS}{\mathrm{S}}

\newcommand{\ru}{\mathrm{u}}
\newcommand{\rg}{\mathrm{g}}
\newcommand{\rd}{\mathrm{d}}
\newcommand{\rk}{\mathrm{k}}

\newcommand{\eqdist}{\ensuremath{\stackrel{\mathrm{d}}{=}}}

\providecommand{\ora}[1]{}
\renewcommand{\ora}[1]{\overrightarrow{#1}}
\DeclareRobustCommand{\SkipTocEntry}[5]{} 

\newcommand{\sgn}{\operatorname{sgn}}

\newcommand{\tree}{t}

\newcommand{\bien}{Bienaym\'e }
\newcommand{\dseq}{\mathtt{d}}
\newcommand{\nseq}{\mathtt{n}}
\newcommand{\wseq}{\mathrm{w}}
\newcommand{\bseq}{\mathtt{b}}

\renewcommand{\H}{{\mathsf{Ht}}}
\renewcommand{\W}{\mathsf{Wd}}
\newcommand{\dfs}{\mathrm{df}}
\newcommand{\bfs}{\mathrm{bf}}

\newtheorem{thm}{Theorem}
\newtheorem{lem}[thm]{Lemma}
\newtheorem{prop}[thm]{Proposition}

\newtheorem{definition}[thm]{Definition}
\newtheorem{remark}[thm]{Remark}

\makenomenclature										
\graphicspath{ {./images/} }
\usepackage{wrapfig}

\numberwithin{equation}{section}
\numberwithin{thm}{section}

\begin{document}
\date{December 31, 2024} 

\title{Tight universal bounds on the height times the width of random trees}

\author{Serte Donderwinkel}
\address{Bernoulli Institute and CogniGron, University of Groningen, Groningen, The Netherlands}
\email{s.a.donderwinkel@rug.nl}

\author{Robin Khanfir}
\address{Department of Mathematics and Statistics, McGill University, Montr\'eal, Canada}
\email{robin.khanfir@mcgill.ca}


\begin{abstract} 
{We obtain assumption-free, non-asymptotic, uniform bounds on the product of the height and the width of uniformly random trees with a given degree sequence, conditioned Bienaym\'e trees and simply generated trees. We show that for a tree of size $n$, this product is $O(n\log n)$ in probability, answering a question by Addario-Berry~\cite{short_fat}. The order of this bound is tight in this generality.}
\end{abstract} 

\maketitle


\section{Introduction}

For a rooted tree $t$, its \emph{size} $\# t$ is the number of nodes of $t$, its \emph{height} $\H(t)$ is the maximum distance of a node of $t$ to the root, and its \emph{width} $\W(t)$ is the maximum generation size in $t$. These statistics give a rough idea of the global shape of the tree, so that the estimation of their orders of magnitude is often the first step towards a finer description of the geometry of the tree. This has motivated a vast body of work, going back to the 1950s, that is devoted to separate studies of the height and of the width for various models of random trees with fixed size~\cite{HarPri59,Riordan,RenSze67,kemp_height,Flajolet82,Tackacs93,ChaMarYor}. In this paper, we focus on the three-way relation between the size, the height, and the width, and we show that they are closely linked for a large class of random trees. This in particular resolves an open question by Addario-Berry~\cite{short_fat}. We first state our results and then present an overview of the literature. After that, we briefly discuss our methods and possible extensions of our work, and pose some open questions.

\subsection{Results}

We begin by focusing on the most elementary random tree model for representing population growth. Let $\mu$ be a probability distribution on $\N=\{0,1,2,\dots\}$. A \emph{$\mu$-Bienaymé tree}, denoted by $T_\mu$, is a random plane tree that represents the family tree of a branching process with offspring distribution $\mu$\footnote{In the literature, the names \emph{Galton--Watson tree} and \emph{Bienaymé--Galton--Watson tree} are also used.}. Then, for each $n$ such that $\P(\#T_{\mu}=n)>0$, let $T_{\mu,n}$ be a $\mu$-\bien tree conditioned to have size $n$.

\noindent In \cite{short_fat}, Addario-Berry asked the following question.
\begin{equation}
\tag{Q}\label{question}
\parbox{\dimexpr\linewidth-4em}{
\strut
Are there examples of offspring distributions for which $\W(T_{\mu,n} )\H(T_{\mu,n})/n$ is much larger than $\log n$ with non-vanishing probability?
\strut
}
\end{equation}

\noindent Our work answers the question from \cite{short_fat} in the negative. In fact, our result is more general, because we obtain an explicit tail-bound for the product of height and width of $T_{\mu,n}$ divided by $n\log n$ that does not depend on $n$ nor $\mu$.

\begin{thm}\label{thm:bien}
   For any $\mu$ and for any $n\geq 3$ such that $\P(\#T_{\mu}=n)>0$, for all $s>0$,
    \[\P(\W(T_{\mu,n})\H(T_{\mu,n})>sn\log n)\leq 230s^{-2/13}.\]
\end{thm}
Note that for any rooted tree $t$ with $\#t=n$, the height times the width of $t$ is at least $n-1$ deterministically. The lower bound is realised if all generation sizes in $t$ (other than the root) are exactly equally large. However, $\H(t)\W(t)$ can be as big as $n^2/4$ when $t$ has one generation of size $n/2$ and all its other generations have size $1$. Our result shows that the product of the height and the width of a random trees is, up to a logarithmic factor, of the same order as the deterministic lower bound.

The following explicit examples show that the order of the bound in Theorem~\ref{thm:bien} is tight. Kortchemski~\cite{Kor15_subcritical} showed that for some non-generic subcritical offspring distributions $\mu$, it holds that $\W(T_{\mu,n} )\H(T_{\mu,n})$ is of order $n\log n$ in probability. Moreover, Addario-Berry~\cite{short_fat} also constructed a critical (i.e.~with mean $1$) offspring distribution with this property. Further such offspring distributions, that all have mean $1$ and are in the domain of attraction of a Cauchy distribution, were given by Addario-Berry, the first author \&  Kortchemski~\cite[Section~3.6]{critical_short_fat}. 

Theorem~\ref{thm:bien} is a consequence of a more general result for \emph{simply generated trees}. The class of simply generated trees is a family of random trees that contains all conditioned Bienaym\'e trees and is defined as follows. Fix non-negative real weights $\wseq=(w_k,k\geq 0)$ with $w_0>0$. Given a finite plane tree $t$, we define the weight of $t$ to be 
\[\wseq(t)=\prod_{v\in t} w_{\rk_v(t)},
\]
where $\rk_v(t)$ stands for the out-degree of $v$ in $t$. 
For positive integers $n$, let
\[Z_n=Z_n(\wseq)=\sum_{\substack{\text{plane trees $t$}\\ \# t=n}} \wseq(\tree)
\]
be the total weight of trees of plane trees of size $n$. 
If $Z_n >0$, then we define a random tree $T_{\wseq,n}$ by setting
\[\P(T_{\wseq,n}=\tree)=\frac{\wseq(\tree)}{Z_n} \]
for plane trees $\tree$ with $\# t=n$. The random tree $T_{\wseq,n}$ is called \emph{the simply generated tree of size $n$ with weight sequence $\wseq$}. In particular, for $\mu$ an offspring distribution, if we set $\wseq=\mu$ then we get that $T_{\wseq,n}\eqdist T_{\mu,n}$ for all $n$ for which $Z_n(\wseq)=\P(\#T_{\mu}=n)>0$. We prove the following result for simply generated trees. 

\begin{thm}\label{thm:sg}
For any $\wseq$ and for any $n\geq 3$ for which $Z_n(\wseq)>0$, for all $s>0$,
    \[\P(\W(T_{\wseq,n})\H(T_{\wseq,n})>sn\log n)\leq 230s^{-2/13}.\]
\end{thm}

We also prove the following result for uniform rooted trees with a given degree sequence. Fix $n\ge 1$ and $\dseq=(d_1,\ldots,d_n)\in \N^n$ that satisfies that $\sum_{i=1}^n d_i=n-1$. A \emph{rooted labelled tree with degree sequence $\dseq$} is an acyclic and connected graph with vertex-set $[n]=\{1,\ldots,n\}$ endowed with a distinguished vertex $r\in[n]$ such that the degree of $r$ is $d_r$ and such that the degree of any other $i\in[n]$ is $d_i+1$. Let $T_{\dseq}$ be a uniformly random rooted labelled tree with degree sequence $\dseq$. 

\begin{thm}\label{thm:dseq}
    For any $n\geq 3$ and $\dseq=(d_1,\dots,d_n)$ with $\sum_{i=1}^n d_i=n-1$, for all $s>0$,
    \[\P(\W(T_{\dseq})\H(T_{\dseq})>sn\log n)\leq 230s^{-2/13}.\]
\end{thm}
\smallskip

Theorems~\ref{thm:bien},~\ref{thm:sg}~and ~\ref{thm:dseq} all follow from a result for uniform plane trees with a given \emph{type sequence}. Let $\nseq=(n_i)_{i\geq 0}$ be a sequence of natural numbers that satisfies that $\sum_{i\geq 0}n_i=1+\sum_{i\geq 0}in_i$. We say that a plane tree $t$ has type $\nseq$ if for all $i\geq 0$, $t$ contains exactly $n_i$ vertices with $i$ children. We call $\sum_{i\ge 0}n_i$ \emph{the size of $\nseq$}. Let $T_\nseq$ be a uniformly random plane tree with type $\nseq$.  We prove the following result. 

\begin{thm}\label{thm:nseq}
    For any $\nseq=(n_i)_{i\geq 0}$ with $\sum_{i\geq 0}n_i=1+\sum_{i\geq 0}in_i=n\geq 3$, for all $s>0$,
    \[\P(\W(T_{\nseq})\H(T_{\nseq})\geq sn\log n)\leq 230 s^{-2/13}.\]
\end{thm}

Theorems \ref{thm:bien} and \ref{thm:sg} follow from \ref{thm:nseq} by averaging, using the observation that, given its type sequence, a conditioned Bienaymé tree (resp.~simply generated tree) is a uniformly random plane tree with this type sequence. Moreover, Theorem~\ref{thm:dseq} follows from Theorem~\ref{thm:nseq} by observing that if, given $T_\dseq$, we forget the labels and sample a uniformly random planar order, we obtain a tree distributed as $T_\nseq$ for $\nseq=(n_i)_{i\geq 0}$ with $n_i=\#\{j:d_j=i\}$. The details of the proofs of Theorems~\ref{thm:bien}, \ref{thm:sg} and \ref{thm:dseq} given Theorem~\ref{thm:nseq} can be found in Section~\ref{sec:all_thms}.

\subsection{Related work}

Let us first give an overview of some of the context of question~(\ref{question}). Let $\mu$ be an offspring distribution on $\N$, let $T_\mu$ be a $\mu$-Bienaymé tree and, if $\p{\#T_\mu=n}>0$, let $T_{\mu,n}$ be a $\mu$-Bienaymé tree conditioned to have size $n$.

When $\mu$ has mean $1$ and finite variance $\sigma^2$, the asymptotic geometry of $T_{\mu,n}$ is very well-understood. First, Kolchin~\cite[Theorem 2.4.3]{Kolchin} showed that $n^{-1/2}\H(T_{\mu,n})$ converges in distribution towards $2\sigma^{-1} M$, where $M$ is the maximum of a normalised Brownian excursion, which turned out to be a consequence of the scaling limit of the whole tree obtained by Aldous~\cite{Aldous93}. Then, Drmota \& Gittenberger~\cite{DrmGit97} proved that in the same setting, $n^{-1/2}\W(T_{\mu,n})$ converges in distribution towards $\sigma M$. It was later shown that these two convergences happen jointly~\cite{ChaMarYor,Jan06} so that $n^{-1}\H(T_{\mu,n})\W(T_{\mu,n})$ converges in distribution to $2M M'$, where $M'\eqdist M$ and the joint law of $(M,M')$ does not depend on $\mu$. Thus, the bigger the variance of the offspring distribution is, the shorter and wider the tree is, but these dependencies on $\sigma^2$ cancel when considering the product of the height and the width of the tree.

This trade-off between height and width of $T_{\mu,n}$ is also understood for some offspring distributions with infinite variance. Indeed, when $\mu$ has mean $1$ and is in the domain of attraction of a stable law of index $\alpha\in(1,2]$, Duquesne~\cite{Duq03} established a scaling limit for $T_{\mu,n}$ and found in particular a slowly varying function $L$ such that $L(n)n^{-1+1/\alpha}\H(T_{\mu,n})$ converges in distribution towards a positive random variable. Then, Kersting~\cite{kersting} proved that $L(n)^{-1}n^{-1/\alpha}\W(T_{\mu,n})$ also converges in distribution. Thus, the bigger $\alpha$, the shorter and wider the tree is, but again, asymptotically, these effects cancel so that $\H(T_{\mu,n})\W(T_{\mu,n})$ is of order $n$ in probability.

As discussed above, the order of the product of height and width can exceed $n$ for more general offspring distributions, but not by much. For example, when $\mu$ has mean $m_\mu<1$ and satisfies that $\mu(k)\sim ck^{-1-\alpha}$ with $c>0$ and $\alpha>1$, Jonsson \& Stefànsson~\cite{JonSte11} showed that the largest degree of $T_{\mu,n}$ is asymptotically equivalent to $(1-m_\mu)n$, an effect that is often called \emph{condensation}. This entails that $\W(T_{\mu,n})$ is of order $n$, whereas Kortchemski~\cite{Kor15_subcritical} proved that $\H(T_{\mu,n})/\log n$ converges in probability to a constant. Furthermore, Addario-Berry, the first author \& Kortchemski~\cite{critical_short_fat} gave asymptotic estimates for the height and the width of $T_{\mu,n}$ when $\mu$ has mean $1$ and $\mu(k)\sim \ell(k)/k^2$ with $\ell$ a slowly varying function, which allowed them to find other offspring distributions for which $\H(T_{\mu,n})\W(T_{\mu,n})$ is of order $n\log n$ in probability. The commonality in these examples is that the offspring distributions have very heavy tails. This causes the width to be of the same order as the largest degree, and the height of the tree is then realised in the tallest subtree pendant to the largest degree vertex. Taking the maximum over many roughly independent subtrees yields the additional $\log n$ factor. Keeping the structure of these trees in mind, it is difficult to imagine an offspring distribution that realises an asymptotically much larger product of the height and the width.

\medskip

Besides the limit theorems discussed above, some works established non-asymptotic tail bounds for the height and width of Bienaymé trees that match the ones of the limiting laws. The first example of such bounds that hold uniformly in $n$ is due to Flajolet, Gao, Odlyzko \& Richmond~\cite{FlaGaoOdlRic93} for uniformly random binary trees. In the general framework of Bienaymé trees, Devroye, Janson \& Addario-Berry~\cite{DevJanLAB13} provided sub-Gaussian tail bounds for $n^{-1/2}\H(T_{\mu,n})$ and $n^{-1/2}\W(T_{\mu,n})$ when the offspring distribution $\mu$ has mean $1$ and finite variance. Similarly, Kortchemski~\cite{Kor15} obtained uniform tail bounds for the height and width of critical Bienaymé trees whose offspring distributions are in the domain of attraction of an $\alpha$-stable law with $\alpha\in(1,2]$. See also Addario-Berry~\cite{LAB12} for bounds on height and width of uniform plane trees with given type sequence that are optimal in a regime corresponding to the finite variance assumption for Bienaymé trees.

The existence of non-asymptotic estimates suggests that convergence under rescaling may not be essential for characterizing the rough orders of magnitude of the height and width of random trees. This intuition has recently given rise to a line of work devoted to the search of strong uniform bounds without any regularity assumptions on the offspring distribution~\cite{short_fat,DonLAB24,BranHamKerLAB22}, in which the present paper fits. In~\cite{short_fat}, Addario-Berry focuses on unconditioned Bienaymé trees but paves the way for future work on conditioned trees by posing the question (\ref{question}). Addario-Berry, Brandenberger, Hamdam \& Kerriou~\cite{BranHamKerLAB22} first prove a tail bound for the height of a typical vertex in a uniform plane tree with fixed type and then deduce nearly tight bounds for the height and width on simply generated trees (thus including conditioned Bienaymé trees). Similar estimates are obtained by Marzouk~\cite{marzouk19} to study scaling limits of uniform plane trees and maps with fixed type. Addario-Berry \& the first author \cite{DonLAB24} provide new non-asymptotic tail bounds for the height of uniformly random plane trees, that correctly match the correct order of the height in many applications. A more precise but more unwieldy inequality is also obtained by Blanc-Renaudie~\cite{BlancR21} in the course of showing scaling limits for uniformly random plane trees with given type.

\subsection{Discussion and open problems}

It is striking that our theorems require no assumptions on the degree distribution at all. Intuitively, this is possible because of the trade-off between the height and width of the random trees that we consider. For example, Theorem 9 in \cite{DonLAB24} morally says that large degrees yield short trees. On the other hand, large degrees make it more likely that the width of a random tree is large. Our proofs exploit this idea. Basically, we use that disjoint subtrees need to share the total mass of the tree, which restricts their heights. Since a large width yields a large number of disjoint subtrees, it must yield a small global height.

This simple intuition is at the core of our relatively elementary proofs, that, unlike previous works that obtain assumption-free uniform bounds, do not require advanced analytical nor probabilistic tools. On the one hand, the papers~\cite{DevJanLAB13,short_fat,BranHamKerLAB22,marzouk19} conduct a fine analysis of spinal decompositions and conditioned random walks associated with random plane trees using strong concentration inequalities. On the other hand, the papers~\cite{DonLAB24,BlancR21} rely on a powerful stick-breaking construction of uniform trees with fixed type yielded by the so-called \emph{Foata--Fuchs bijection}~\cite{FoataFuchs,FoataFuchs_stick-break}, and then delicately inspect their probabilistic properties. While we also employ spinal decompositions and encodings of trees by walks, the control we need on them does not require more than the second moment method. Furthermore, we do not use the Foata--Fuchs bijection at all. Instead, we make extensive use of the exchangeability in the random trees by constructing various law-preserving transformations.
\smallskip

In return for being able to obtain the tight bounds $n\log n$ by elementary means, the tail bound given by Theorem~\ref{thm:nseq} is far from optimal. While the exponent $2/13$ is the best possible without changing any inequalities we used, any small improvement (using, for example, a third moment instead of a second) would increase it. Improving $2/13$ to a number exceeding $1$ would entail that the first moment of the product of the height and width of a tree of size $n$ is uniformly bounded on the scale $n\log n$ across all models for random trees that we consider. In fact, we believe that Theorem~\ref{thm:nseq} should still hold with a bound that decays (at least) exponentially. We leave such improvements for future work. 

In addition to the search of better tail bounds, we state some open questions naturally suggested by our work. Recall that $T_{\mu,n}$ stands for a $\mu$-Bienaymé tree conditioned to have size $n$ and $T_\mu$ stands for an unconditioned $\mu$-Bienaymé tree. 
\begin{enumerate}
    \item We believe that $S=\sup \E{\I{\# T_\mu<\infty}\H(T_\mu)\W(T_\mu)(\# T_\mu)^{-1}}$, where the supremum is taken among all probability measures $\mu$ on $\N$, is finite. We would further like to compute its value and to know whether it is achieved, and if so by what offspring distribution.
    \item We proved that the order of $\H(T_{\mu,n})\W(T_{\mu,n})$ is always between $n$ and $n\log n$. What are the conditions on the offspring distribution $\mu$ for being in one of the limit cases? Namely, when do we have $\H(T_{\mu,n})\W(T_{\mu,n})=\Theta(n)$ in probability and when do we have $\H(T_{\mu,n})\W(T_{\mu,n})=\Theta(n\log n)$ in probability?   
    \item For any slowly varying function $1\ll f(n)\ll\log(n))$, is there an offspring distribution $\mu$ such that $\H(T_{\mu,n})\W(T_{\mu,n})=\Theta(n f(n))$?
\end{enumerate}

\subsection{Structure of the paper}

The remainder of the paper is structured as follows. Section~\ref{sec:ptrees} presents definitions used throughout the paper, as well as the main propositions used in the proof of Theorem~\ref{thm:nseq} and how they imply the result. In this section, we also show how Theorems~\ref{thm:bien}, \ref{thm:sg}, and \ref{thm:dseq} follow from Theorem~\ref{thm:nseq}. Each of the remaining sections of the paper is devoted to proving one of the main ingredients of the proof of Theorem~\ref{thm:nseq}. Section~\ref{sec:H2vsH} rules out that a uniformly random tree with a given type looks like a line. Section~\ref{sec:coding} presents walks that encode plane trees and controls their ranges. Section~\ref{sec:height_spine} uses a stochastic domination argument to bound the heights of subtrees stemming from one side of a spine. Finally, Section~\ref{sec:symmetrization} introduces shuffling transformations and uses them to expand the results of Section~\ref{sec:height_spine} to uniform bounds across the whole tree, as long as the tree is not line-like.

\section{Framework and main ideas}\label{sec:ptrees}

\subsection{Words formalism for plane trees}

We denote by $\N^*=\{1,2,3,\ldots\}$ the set of positive integers and by $\U$ the set of finite words written with the alphabet $\N^*$. Namely,
\[\U=\bigcup_{\ell\geq 0}(\N^*)^\ell\quad\text{ with the convention }\quad(\N^*)^0=\{\varnothing\}.\]
As a set of words, $\U$ is totally ordered by the \emph{lexicographic order} $\leq$, so that $\varnothing<(1)<(1,2)<(2)$ for example. For two words $(u_1,\ldots,u_\ell)$ and $(v_1,\ldots,v_m)$, we write $u*v=(u_1,\ldots,u_\ell,v_1,\ldots,v_m)$ for their concatenation. We also denote by $|u|=\ell$ the length of $u$, with $|\varnothing|=0$, which we call \emph{the height of $u$}. Moreover, if $u\neq\varnothing$ then we call $\overleftarrow{u}:=(u_1,\ldots,u_{\ell-1})$ \emph{the parent of $u$}. If $u,v\in \U\setminus\{\varnothing\}$ have the same parent and if $u\leq v$, we say that $u$ is an \emph{older sibling} of $v$ and that $v$ is a \emph{younger sibling} of $u$.  We also define the \emph{genealogical order} $\preceq$ on $\U$, which is the partial ordering generated by the covering relation\footnote{For a partially ordered set $(\cP,\prec)$, $y\in \cP$ covers $x\in \cP$ if $x\prec y$ and for all $z\in \cP$, if $x\preceq z \preceq y$ then $z=x$ or $z=y$.} $u\in t$ covers $v\in t$ if and only if $v=\overleftarrow{u}$. Then $u\preceq v$ if and only if there is a $w\in \mathbb{U}$ so that $v=u*w$. If $u\preceq v$ we say that $u$ is an \emph{ancestor of $v$}. We also write $u \prec v$ to express that $u \preceq v$ and $u \ne v$.  For $u,v\in t$, the \emph{most recent common ancestor of $u$ and $v$}, denoted by  $u\wedge v$, is the maximal element in the genealogical order that is an ancestor of both $u$ and $v$. 

\begin{figure}
\centering
\includegraphics[page=1]{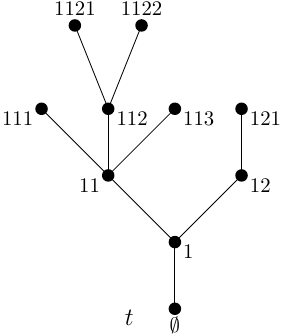}
\caption{A depiction of a plane tree $t$. The type of $t$ is $\nseq(t)=(5,2,2,1,0,0,\dots)$. The height of $t$ is $\H(t)=4$ (realised by vertices $1121$ and $1122$) and its width is $\W(t)=4$ (realised by the generation at distance $3$ from the root). The right spinal weight of vertex $112$ in $t$ is $\rS_{112}^{\rd}(t)=\#\{113,12\}=2$ and the left spinal weight of vertex $112$ is $\rS_{112}^{\rg}(t)=\#\{111\}=1$. Its spinal weight is $\rS_{112}(t)=\rS_{112}^{\rd}(t)+\rS_{112}^{\rg}(t)=3$. The second order height of $t$ is $\H^{(2)}(t)=2$, where the maximum of $\min\big(|u|-|u\wedge v|,|v|-|u\wedge v|\big)$ is realised when either $\{u,v\}=\{1121,121\}$ or $\{u,v\}=\{1122,121\}$. \label{fig:planetree}}
\end{figure}
\begin{definition}
\label{def:tree}
A \emph{plane tree} is a finite subset of $\U$ satisfying the following conditions:
\begin{compactenum}
    \item[$(a)$] $\varnothing\in t$,
    \item[$(b)$] for all $u\in t$, if $u\neq\varnothing$ then $\overleftarrow{u}\in t$,
    \item[$(c)$] for all $u\!\in\! t$, there is an integer $\rk_u(t)\!\geq\! 0$ such that $u\! *\! (i)\in t \!\Leftrightarrow\! i\in [\rk_u(t)]$ for all $i\! \in\! \N^*$.
\end{compactenum}
\end{definition}
\noindent
We let $\mathcal{T}$ denote the set of plane trees. Plane trees are our main object of study in this work, so we will often say \emph{tree} to mean \emph{plane tree}. As introduced earlier, a \emph{type sequence of size $n$} is a sequence $\nseq=(n_i)_{i\geq 0}$ such that $\sum_{i\geq 0}n_i=1+\sum_{i\geq 0}in_i$. The \emph{type} of a tree $t$ is $\nseq(t)=( n_i(t),i\ge 0)$ with $n_i(t)=\#\{u\in t:\rk_u(t)=i\}$. See Figure~\ref{fig:planetree}. We denote by $\cT(\nseq)$ the set of trees with type $\nseq$.

Let $t$ be a tree. We call $\varnothing$ the \emph{root} of $t$. For $u\in t$ we call $\{u\! *\! (i): i \in [\rk_u(t)]\}$ the \emph{children of $u$} in $t$. A \emph{leaf of $t$} is a vertex $u\in t$ such that $\rk_u(t)=0$. For $u\in t$, we let $\theta_u (t)=\{v:u*v\in t\}$ be the \emph{subtree rooted at $u$}. It is easy to check that $\theta_u t\in \cT$. The \emph{height} $\H(t)$ and the \emph{width} $\W(t)$ of a tree $t$ are defined by
\[\H(t)=\max_{u\in t}|u|\quad\text{ and }\quad \W(t)=\max_{h\geq 0}\#\{u\in t\, :\, |u|=h\}.\]
See Figure~\ref{fig:planetree}.  Observe that if $t$ has $n$ vertices and $ n_0$ leaves then
\begin{equation}
\label{basic_bound_H_W}
\H(t)\leq n-1\quad\text{ and }\quad\W(t)\leq  n_0.
\end{equation}
Indeed, first, the bound on $\H(t)$ readily follows from condition $(b)$ of Definition~\ref{def:tree} of (plane) trees. Second, for any $h\geq 0$, there is a leaf $v_u\in\theta_u t$ for all $u\in t$ with $|u|=h$. So, the vertices $u*v_u$ for $u\in t$ with $|u|=h$ are distinct leaves of $t$, which yields that $\#\{u\in t:|u|=h\}\leq  n_0$ for any $h\geq 0$.

\subsection{Spines in plane trees}
\label{sec:construction_around_spine}

Let $t$ be a tree endowed with a distinguished vertex $u\in t$. The ancestors of $u$ form a path from $u$ to the root of tree, which we call \emph{spine}, from which stem disjoint subtrees. This \emph{spinal decomposition} represents one of the main tool of this paper because it allows us to study the width and the height simultaneously. Basically, the vertices attached to the spine are not ancestrally related, just like the elements of a generation $\{v\in t:|v|=h\}$, and it turns out that maximising the size of this set over the different possible spines yields a number of the same order as the width of the tree. Moreover, unlike a generation, these vertices are still associated to a branch of the tree, whose maximum possible length is the height $\H(t)$.
\smallskip

To keep track of the number of trees attached to the spine but not stemming from $u$, we define the \emph{spinal weights of $(t,u)$} as
\begin{align*}
\rS_u(t)&=\#\{v\in t\setminus\{\varnothing\}\, :\, \overleftarrow{v}=u\wedge v\text{ and }\overleftarrow{v}\neq u\},\\
\rS_u^{\rd}(t)&=\#\{v\in t\setminus\{\varnothing\}\, :\, u<v,\overleftarrow{v}=u\wedge v\text{ and }\overleftarrow{v}\neq u\},\\
\rS_u^{\rg}(t)&=\rS_u(t)-\rS_u^{\rd}(t),
\end{align*}
so that $\rS_\varnothing(t)=\rS_\varnothing^{\rd}(t)=\rS_\varnothing^{\rg}(t)=0$. See Figure~\ref{fig:planetree}. Viewing $t$ as a family tree, $\rS_u$ counts the number of siblings of $u$ and its ancestors, $\rS_u^{\mathrm{d}}(t)$ counts the number of younger siblings of $u$ and its ancestors, and $\rS_u^{\mathrm{g}}$ counts the number of older siblings of $u$ and its ancestors. In other words, we have the identities
\begin{equation}
\label{spinal-weights_as_sums}
\rS_u(t)=\sum_{\substack{v\in\mathbb{U},j\geq 1\\ v*(j)\preceq u}}\big(\rk_v(t)-1\big),\quad \rS_u^{\rg}(t)=\sum_{\substack{v\in\mathbb{U},j\geq 1\\ v*(j)\preceq u}}(j-1),\quad \rS_u^{\rd}(t)=\sum_{\substack{v\in\mathbb{U},j\geq 1\\ v*(j)\preceq u}}\big(\rk_v(t)-j\big).
\end{equation}

We will also be interested in the maximum height of the pendant trees off the spine, in particular for $u$ that realise the height of the whole tree. We introduce the \emph{second-order height} of the tree $t$ as follows:
\[\H^{(2)}(t)=\max_{u,v\in t}\min\big(|u|-|u\wedge v|,|v|-|u\wedge v|\big).\]
See Figure~\ref{fig:planetree}. The following proposition asserts that $\H^{(2)}(t)$ is the maximum height of a subtree of $t$ pendant off a path from the root to a vertex with maximum height.  We will later show that in the classes of random trees that we consider, typically, the two characteristics $\H(t)$ and $\H^{(2)}(t)$ are of the same order, which will allow us to bound $\H(t)$ by controlling $\H^{(2)}(t)$ via spinal decompositions.

\begin{prop}
\label{second-height_alt}
Let $t$ be a tree and let $\ru\in t$. If $|\ru|=\H(t)$, then
\begin{align}
\label{second-height_alt1}
\H^{(2)}(t)&=\max_{v\in t}|v|-|\ru\wedge v|\\
\label{second-height_alt2}
&=\min_{u\in t}\max_{v\in t}|v|-|u\wedge v|.
\end{align}
\end{prop}

\begin{proof}
We have
\begin{equation}
\label{second-height_alt_step1}
\H^{(2)}(t)\geq \max_{v\in t}\min \big(|\ru|-|\ru\wedge v|,|v|-|\ru\wedge v|\big) = \max_{v\in t}|v|-|\ru\wedge v|,
\end{equation}
where the equality follows from $|\ru| \geq |v|$ for any $v\in t$.
Conversely, let $u\in t$ be any vertex. By definition of $\H^{(2)}(t)$, there are $v,w\in t$ such that $\H^{(2)}(t)+|v\wedge w|\leq |v|,|w|$. Both $u\wedge v$ and $u\wedge w$ are ancestors of the same vertex $u$ so we can assume that $u\wedge v\preceq u\wedge w$ without loss of generality. It then follows that $u\wedge v=u\wedge (v\wedge w)\preceq v\wedge w$ and so that $\H^{(2)}(t)\le |v|-|u\wedge v|$. We have thus proved that
\begin{equation}
\label{second-height_alt_step2}
\H^{(2)}(t)\leq \min_{u\in t}\max_{v\in t}|v|-|u\wedge v|.
\end{equation}
Since the right-hand side of (\ref{second-height_alt_step1}) is larger or equal to the right-hand side of (\ref{second-height_alt_step2}), we readily obtain the desired identities (\ref{second-height_alt1}) and (\ref{second-height_alt2}).
\end{proof}

\subsection{Overview of the proof of the main theorem}\label{sec:overview}

Here, we explain the global strategy of the proof of Theorem~\ref{thm:nseq} by presenting the key steps involved. We then show how these key ingredients, whose proofs are postponed to later in the paper, imply Theorem~\ref{thm:nseq}. Fix a type sequence $\nseq=( n_i)_{i\geq 0}$ of size $n\geq 3$, and let $T$ be uniformly distributed on $\cT(\nseq)$ so that $T$ is a uniformly random tree with type $\nseq$.
\smallskip

We first want to replace the height $\H(T)$ and the width $\W(T)$ in Theorem~\ref{thm:nseq} by the second order height $\H^{(2)}(T)$ and the maximal right spinal weight $\max_{u\in T}\rS^\rd_u(T)$ respectively. The latter quantities are both expressed in terms of subtrees pendant off spines and are therefore easier to track jointly. 

\smallskip
Thus, a first big step of the proof of Theorem~\ref{thm:nseq} is to compare the height with the second-order height. Informally, if $\H^{(2)}(t)$ is small compared to $\H(t)$, then $t$ looks like a line (with small trees branching off it). The following proposition asserts that this is unlikely to occur for a uniformly random tree with given type (provided it has many leaves).
\begin{prop}
\label{prop:not_line}
For all $\varepsilon >0$, if $\varepsilon^{1/3} n_0\geq 2$ then it holds that
\[\P\left(\varepsilon\H(T)\geq 1\, ;\, \varepsilon\H(T)> \H^{(2)}(T)\right)\leq 64\varepsilon^{1/3}.\]
\end{prop}
To obtain this result, we construct and study several type-preserving maps that swap, cut, and re-graft some subtrees of an input tree. Thanks to these transformations, we can associate each line-like tree to many different fork-like trees, for which the second-order height is similar to the height. The proof of Proposition~\ref{prop:not_line} can be found in Section~\ref{sec:H2vsH}.
\smallskip

A next step of the proof of Theorem~\ref{bound_product_right-weight-height} is a comparison of the maximum right-spinal weight with the width, which is the content of the following result.
\begin{prop}
\label{width_vs_Luka}
For all $\varepsilon>0$, if $\varepsilon^{4/3}\sqrt{ n_0-1}\geq 2^6$ then 
\[\P\big(\varepsilon\W(T)\geq \max_{u\in T} \rS_u^{\rd}(T)\big)\leq \sqrt[3]{2^{22}}\, \varepsilon^{2/3}.\]
\end{prop}
The proof of Proposition~\ref{width_vs_Luka} relies on two standard codings of plane trees by discrete walks, derived from the breadth-first search and the depth-first search. We use that their maxima are closely related to $\W(T)$ and $\max_{u\in T}\rS_u^{\rd}(T)$ respectively. For the uniformly random tree $T$ with fixed type, these two walks have the same distribution and can be obtained from an exchangeable bridge via the Vervaat transform. We use the method of moments to study the lower and upper tails of the range of such an exchangeable bridge, and find that it has roughly deterministic order, so that both $\W(T)$ and $\rS_u^{\rd}(T)$ are of this order too. This then implies Proposition~\ref{width_vs_Luka}. The proof is presented in Section~\ref{sec:coding}. 
\smallskip

The next key building block of the proof of Theorem~\ref{thm:nseq} is the following proposition that uniformly relates, for any choice of spine, the number of subtrees stemming from the right of the spine (i.e.~the right-spinal weight) to their maximum height.
\begin{prop}
\label{bound_product_right-weight-height}
For all $s>0$, it holds that 
\[\P\left(\max_{\substack{u,v\in T\\ u\leq v}}(|v|-|u\wedge v|)\rS_u^{\rd}(T)\geq sn\log n\right)\leq 3n^{2-s/2}.\]
\end{prop}
To show Proposition~\ref{bound_product_right-weight-height}, we adapt an argument of Addario-Berry \& the first author \cite{DonLAB24}. We show that among all type sequences of \emph{forests} with a fixed number of vertices and components, the type sequence that yields the stochastically tallest uniform forest contains no vertices with more than $1$ child. The components of this random forest are paths, and their maximum height is distributed as the largest part of a uniformly random composition of a given integer. This yields the factor $\log n$, which in fact exclusively stems from Proposition~\ref{bound_product_right-weight-height}. We show Proposition~\ref{bound_product_right-weight-height} in Section~\ref{sec:height_spine}.
\smallskip

Our proof of Proposition~\ref{bound_product_right-weight-height} also relies on an exploration of the tree from left to right that introduces an artificial asymmetry within the statement. Indeed, the planar order of a uniformly random tree with a given type is itself uniform. We overcome the asymmetry in two stages. First, we justify that we can replace $\rS_u^{\rd}(T)$ with $\rS_u(T)$. Informally, the vertices that attach to the spine of node $u$ are all equally likely to attach to the left or right, so if $\rS_u(T)$ is large then the left and right spinal weights are likely to be of the same order. This is the content of the following result. 
\begin{prop}
\label{left-right_spine_weights}
For all $\varepsilon >0$, it holds that
\[\P\Big(\max_{u\in T}\rS_u(T)\geq \varepsilon^{-1}\, ;\, \max_{u\in T}\rS_u(T)\geq\varepsilon^{-1}\max_{u\in T}\min\big(\rS_u^{\rd}(T),\rS_u^{\rg}(T)\big)\Big)\leq 2\sqrt{8\varepsilon}.\]
\end{prop}
Second, we remove the restriction $u\leq v$ within Proposition~\ref{bound_product_right-weight-height} by using that the tallest subtree pendant off the spine is equally likely to be on the left or the right. This allows us to bring the second-order height into play as follows.
\begin{prop}
\label{left-right_spine_height}
For all $s>0$, it holds that
\[\P\left(\H^{(2)}(T)\max_{u\in T}\min\big(\rS_u^{\rd}(T),\rS_u^{\rg}(T)\big)\geq s\right)\leq 2\P\left(\max_{\substack{u,v\in T\\ u\leq v}} (|v|-|u\wedge v|)\rS_u^{\rd}(T)\geq s\right).\]
\end{prop}
We prove Propositions~\ref{left-right_spine_weights} and \ref{left-right_spine_height} in Section~\ref{sec:symmetrization} by exploiting the randomness of the planar order of $T$ with law-preserving shuffling operations. 
\smallskip

Assuming the five Propositions~\ref{prop:not_line}, \ref{width_vs_Luka}, \ref{bound_product_right-weight-height}, \ref{left-right_spine_weights} and \ref{left-right_spine_height}, we can now prove Theorem~\ref{thm:nseq}.

\begin{proof}[Proof of Theorem~\ref{thm:nseq}]
If $0\leq s\leq 2^{7\cdot 13/2}$ then $230 s^{-2/13}\geq 1$, which is an upper bound for any probability. Thus, we assume throughout this proof that $s\geq 2^{7\cdot 13/2}$. Recalling from (\ref{basic_bound_H_W}) the basic inequalities $\H(T)\leq n$ and $\W(T)\leq  n_0$, if $s> n_0$ then $sn\log n> n_0 n\log 3\geq \W(T)\H(T)$ almost surely. This means that we can also assume that $s\leq n_0$.

Let $\varepsilon_1,\varepsilon_2,\varepsilon_3>0$ that we will choose later. Recall that $\rS_u(T)=\rS_u^{\rg}(T)+\rS_u^{\rd}(T)\geq\rS_u^{\rd}(T)$. By comparing $\H(T)$ with $\H^{(2)}(T)$ on the one hand, and $\W(T)$ with $\max_{u\in T} \rS_u^{\rd}(T)$ and then $\max_{u\in T}\rS_u(T)$ with $\max_{u\in T}\min\big(\rS_u^{\rd}(T),\rS_u^{\rg}(T)\big)$ on the other hand, we obtain
\begin{align*}
\P\big(\H(T)\geq \varepsilon_3^{-1}\, ;\,& \W(T)\geq \varepsilon_1^{-1}\varepsilon_2^{-1}\, ;\, \W(T)\H(T)\geq s n\log n\big)\\
\leq \quad&\P\big(\varepsilon_3\H(T)\geq 1\, ;\, \varepsilon_3\H(T)>\H^{(2)}(T)\big)\\
+&\P\big(\varepsilon_1\W(T)\geq \max_{u\in T} \rS_u^{\rd}(T)\big)\\
+&\P\Big(\max_{u\in T}\varepsilon_2\rS_u(T)\geq 1\, ;\, \max_{u\in T}\varepsilon_2\rS_u(T)\geq \max_{u\in T}\min\big(\rS_u^{\rd}(T),\rS_u^{\rg}(T)\big)\Big)\\
+&\P\Big(\H^{(2)}(T)\max_{u\in T}\min\big(\rS_u^{\rd}(T),\rS_u^{\rg}(T)\big)\geq \varepsilon_1\varepsilon_2\varepsilon_3 s n\log n\Big)\, .
\end{align*}
Furthermore, Proposition~\ref{left-right_spine_height} asserts that
\begin{multline*}
\P\Big(\H^{(2)}(T)\max_{u\in T}\min\big(\rS_u^{\rd}(T),\rS_u^{\rg}(T)\big)\geq \varepsilon_1\varepsilon_2\varepsilon_3 s n\log n\Big)\\
\leq 2\P\Big(\max_{\substack{u,v\in T\\ u\leq v}} (|v|-|u\wedge v|)\rS_u^{\rd}(T)\geq \varepsilon_1\varepsilon_2\varepsilon_3 s n\log n\Big).
\end{multline*}
Next, if $\varepsilon_1^{4/3}\sqrt{ n_0-1}\geq 2^{6}$ and $\varepsilon_3^{1/3} n_0\geq 2$, then we can apply Propositions~\ref{prop:not_line}, \ref{width_vs_Luka}, \ref{left-right_spine_weights},  and \ref{bound_product_right-weight-height} to bound the terms of the right-hand side as follows:
\begin{multline}
\label{almost_final_step}
\P\big(\H(T)\geq \varepsilon_3^{-1}\, ;\, \W(T)\geq \varepsilon_1^{-1}\varepsilon_2^{-1}\, ;\, \W(T)\H(T)\geq s n\log n\big)\\
\leq \sqrt[3]{2^{22}}\varepsilon_1^{2/3} + 2\sqrt{8\varepsilon_2} + 64\varepsilon_3^{1/3} + 6 n^{2-\varepsilon_1\varepsilon_2\varepsilon_3 s/2}.
\end{multline}

Now, we choose the parameters $\varepsilon_1,\varepsilon_2,\varepsilon_3$ as follows:
\[\varepsilon_1=\tfrac{1}{\sqrt{2}} s^{-3/13}, \quad \varepsilon_2=2^{7} s^{-4/13},\quad \varepsilon_3=\tfrac{1}{2^3}s^{-6/13}.\] Thanks to the assumption that $ n_0\geq s\geq 2^{7\cdot 13/2}$, we verify that 
\[\varepsilon_1^{4/3}\sqrt{ n_0-1}\geq 2^{-2/3} n_0^{-4/13}\sqrt{ n_0/2}=2^{-7/3} n_0^{5/26}\geq 2^{-7/3+7\cdot 5/4}=2^{77/12}\geq 2^6,\]
and that $\varepsilon_3^{1/3} n_0\geq \tfrac{1}{2}  n_0^{-2/13} n_0\geq 2$. This allows us to write (\ref{almost_final_step}) with our choice for $\varepsilon_1,\varepsilon_2,\varepsilon_3$. We also compute $\varepsilon_1\varepsilon_2\varepsilon_3 s/2-2=2^{2+1/2}-2\geq 2/13$ and we note that $n\geq  n_0\geq s$. Therefore, (\ref{almost_final_step}) entails that
\begin{equation}
\label{almost_almost_final_step}\P\big(\H(T)\geq \varepsilon_3^{-1}\, ;\, \W(T)\geq \varepsilon_1^{-1}\varepsilon_2^{-1}\, ;\, \W(T)\H(T)\geq s n\log n\big)\leq 230s^{-2/13}.
\end{equation}

Finally, if $sn\log n \leq \W(T)\H(T)$, then $s\leq \H(T)$ because $\W(T)\leq n$, and $s\leq \W(T)$ because $\H(T)\leq n$. Thus, the inequality $sn\log n \leq \W(T)\H(T)$ implies that $\H(T)\geq s^{7/13} \cdot s^{6/13}\geq \varepsilon_3^{-1}$ and $\W(T)\geq s^{6/13}\cdot s^{3/13}s^{4/13}\geq \varepsilon_1^{-1}\varepsilon_2^{-1}$, since $s\geq 2^{7\cdot 13/2}$. Hence,
\[\P\big(\W(T)\H(T)\! \geq\! s n\log n\big)\leq \P\big(\H(T)\!\geq\! \varepsilon_3^{-1}\, ;\, \W(T)\! \geq\! \varepsilon_1^{-1}\varepsilon_2^{-1}\, ;\, \W(T)\H(T)\! \geq\!  s n\log n\big),\] which allows us to conclude the proof with (\ref{almost_almost_final_step}).
\end{proof}

\subsection{Proofs of the other theorems}
\label{sec:all_thms}
We now show how Theorems~\ref{thm:bien}, \ref{thm:sg} and \ref{thm:dseq} follow from Theorem~\ref{thm:nseq}. We earlier observed that Bienaymé trees are a special case of simply generated trees, so Theorem~\ref{thm:sg} implies Theorem~\ref{thm:bien}. We thus start by proving Theorem~\ref{thm:sg}. 

\begin{proof}[Proof of Theorem~\ref{thm:sg} using Theorem~\ref{thm:nseq}]
Fix a weight sequence $\wseq$ and an $n\ge 3$ such that $Z_n(\wseq)>0$. Let $\nseq=(n_i)_{i\ge 0}$ be a type sequence of size $n$. We claim that the law of $T_{\wseq,n}$ assigns equal mass to all trees with type sequence $\nseq$. Indeed, for $t$ a tree with type $\nseq$, 
\[\p{T_{\wseq,n}=t}=\frac{1}{Z_n(\wseq)}\prod_{v\in t} w_{\rk_v(t)}=\frac{1}{Z_n(\wseq)}\prod_{i\ge 0} w_i^{n_i}\]
which does not depend on $t$. Thus, conditional on $\nseq(T_{\wseq,n})=\nseq$, the tree $T_{\wseq,n}$ is distributed as $T_\nseq$. We apply Theorem~\ref{thm:nseq} to find that for all $s>0$, 
\begin{align*}
&\P(\W(T_{\wseq,n})\H(T_{\wseq,n})>sn\log n)\\
&\quad=\sum_{\nseq}\p{\nseq(T_{\wseq,n})=\nseq }\p{\W(T_{\wseq,n})\H(T_{\wseq,n})>sn\log n\mid \nseq(T_{\wseq,n})=\nseq }\\
&\quad=\sum_{\nseq}\p{\nseq(T_{\wseq,n})=\nseq }\p{\W(T_{\nseq})\H(T_{\nseq})>sn\log n }\\
&\quad \le \sum_{\nseq}\p{\nseq(T_{\wseq,n})=\nseq } 230s^{-2/13}=230s^{-2/13},
\end{align*}
where the sums run over all type sequences $\nseq$ for which $\p{\nseq(T_{\wseq,n})=\nseq}>0$.
\end{proof}

\begin{proof}[Proof of Theorem~\ref{thm:dseq} using Theorem~\ref{thm:nseq}]
Let $\dseq=(d_1,\dots,d_n)$ with $\sum_{i=1}^nd_i=n-1$ be a degree sequence of a tree of size $n$ and define $n_i=\#\{j:d_j=i\}$, so that $\nseq=(n_i)_{i\ge 0}$ is a type sequence of size $n$. We will show that 
\begin{equation}\label{eq:coupling_dseq_nseq}(\H(T_\dseq),\W(T_\dseq))\eqdist(\H(T_\nseq),\W(T_\nseq)),\end{equation}
so that Theorem~\ref{thm:dseq} then follows from Theorem~\ref{thm:nseq}. In this proof, by a {\em labelled plane tree} we mean a 
plane tree $\mathbf{t}$ of size $n$ whose vertices are labelled by the integers $[n]$. From a labelled plane tree, we may obtain a plane tree by ignoring the vertex labels, and we may obtain a labelled rooted tree with vertex set $[n]$ by ignoring the orderings. 

Let $\mathbf{T}_\nseq$ be the random labelled plane tree obtained from $T_\nseq$ by, for each $i\ge 0$, assigning the labels $\{j:d_j=i\}$ to its $n_i$ vertices with $i$ children uniformly at random (without replacement). Then, let $\widehat{T_\nseq}$ be the labelled rooted tree obtained from $\mathbf{T}_\nseq$ by forgetting its planar ordering. We claim that $\widehat{T_\nseq}\eqdist T_\dseq$. Observe that for $t$ a labelled rooted tree with degree sequence $\dseq$,  
\[\p{\widehat{T_\nseq}=t}=\sum_{\mathbf{t}}\p{\mathbf{T}_\nseq=\mathbf{t}},\]
where the sum is over labelled plane trees $\mathbf{t}$ whose underlying (unordered) labelled rooted tree is $t$. This sum has $\prod_{i=1}^n d_i!$ summands, because for each $i\in [n]$ we can assign $d_i!$ possible orderings to its children. Moreover, for any $\mathbf{t}$ included in the sum, writing $t'$ for its underlying (unlabelled) plane tree, we see that $t'$ must have type $\nseq$ because $t$ has degree sequence $\dseq$, so 
\[\p{\mathbf{T}_\nseq=\mathbf{t}}= \frac{1}{\prod_{i\ge 0} n_i!}\p{T_\nseq=t'}= \frac{1}{\#\cT(\nseq)\prod_{i\ge 0} n_i!},\]
which does not depend on $\mathbf{t}$. Thus, $\P(\widehat{T_\nseq}=t)$ does not depend on $t$, so that $\widehat{T_\nseq}\eqdist T_\dseq$ as claimed. Finally, we see that $\W(T_\nseq)=\W(\widehat{T_\nseq})$ and $\H(T_\nseq)=\H(\widehat{T_\nseq})$, so that $(T_\nseq, \widehat{T_\nseq})$ is a coupling of $T_\nseq$ and $T_\dseq$ under which both trees have the same height and width. This implies \eqref{eq:coupling_dseq_nseq} and thus concludes the proof. 
\end{proof}

\section{A random tree does not look like a line} 
\label{sec:H2vsH}

Let $t$ be a tree. Recall that $\H(t)=\max_{u\in t}|u|$ stands for its height. In this section, we are interested in comparing the height to the second-order height introduced in Section~\ref{sec:construction_around_spine}:
\[\H^{(2)}(t)=\max_{u,v\in t}\min\big(|u|-|u\wedge v|,|v|-|u\wedge v|\big).\]
More precisely, our goal is to show Proposition~\ref{prop:not_line}, i.e.~that if $\nseq=(n_i)_{i\geq 0}$ is a type sequence of size $n\geq 3$ and if $T$ is uniformly distributed on $\cT(n)$, then $\H(T)$ and $\H^{(2)}(T)$ have roughly the same order with high probability.

The intuition behind the desired proposition is as follows. If a tree $t$ looks like a line then swapping the highest quarter of the line with some leaf in a pendant tree off the lower half of the line yields many possible trees, all of the same type as $t$, whose second-order heights are close to their heights. In other words, even if $T$ is uniformly built from gluing only one line onto another, it is more likely to look like a fork than a line. (In our proof, we will optimise the quantities $1/4$ and $1/2$.)

To follow this idea, we need to choose the subtree of $t$ that we will cut and re-graft elsewhere. We do this by defining $\ru(t)$ as the minimal vertex of $t$ (for the lexicographic order) such that $|\ru(t)|=\H(t)$, the spine then being the ancestral lineage of $\ru(t)$. Then, (\ref{second-height_alt1}) in Proposition~\ref{second-height_alt} gives $\H^{(2)}(t)=\max_{v\in t}|v|-|v\wedge\ru(t)|$. For any $\lambda\in[0,1]$, set
\begin{equation}
\label{def_rL_lambda}
\rL_\lambda(t)=\#\{v\text{ leaf of }t\, :\, |v\wedge \ru(t)|\geq (1-\lambda)|\ru(t)|\}.
\end{equation}
to keep track of the number of possible choices for the grafting position. Our strategy relies on the fact that if $\lambda$ is small then $\rL_\lambda(t)$ is small compared to the total number of leaves, i.e.~$\rL_1(t)$, as stated by the following lemma. This result, informally, implies that there are many leaves in the lower portion of a line-like tree to choose from, resulting in many fork-like trees for each line-like tree.

\begin{lem}
\label{lem:half_leaves_low}
Let $\nseq=( n_i)_{i\geq 0}$ be a type sequence and let $T$ be uniformly distributed on $\cT(\nseq)$. For all $\delta,\lambda\in(0,1/2)$, if $\delta n_0\geq 2$ then it holds that
\[\P\big(\H(T)\geq 1/\lambda\, ;\, \rL_\lambda(T)>  \delta n_0\big)\leq 12\lambda/\delta^2.\]
\end{lem}

\begin{proof}[Proof of the lemma]

The idea is to perform a random circular shift on the ancestral lineage of $\ru(T)$ so that the position of the most recent common ancestor of a given leaf becomes uniform on the spine. Unfortunately, this transformation will not preserve the law of $T$ because it may increase (but not decrease) the height. We overcome this issue by using a uniform random leaf of $T$ instead of $\ru(T)$. Hence, we begin by relating the event $E=\{\H(T)\geq 1/\lambda\, ;\, \rL_\lambda(T)> \delta n_0\}$ to properties of uniformly random leaves.

\begin{figure}
\centering
\includegraphics[page=1,width=0.3\textwidth]{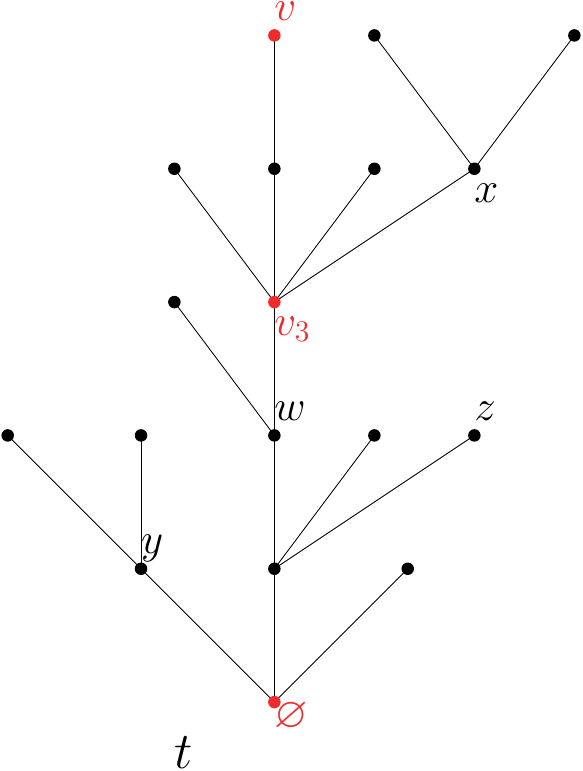}
\hspace{1em}
\includegraphics[page=2,width=0.3\textwidth]{cyclic_shift}
\caption{\label{fig:cyclic} A depiction of the cyclic shift of $t$ along the ancestral line of $v$ for $i=3$. The vertices $v, v_3,\varnothing$ in $t$ are marked, and so are their images under $\psi$ in $\psi(t)$. The same holds for four arbitrary vertices $w,x,y,z$ in $t$.}
\end{figure}
Conditionally given $T$, we uniformly draw without replacement two leaves $V,V'$ of $T$. Note that this is always possible because the number of leaves of $T$ is $ n_0\geq 2$. Let us write $A=\{|V\wedge \ru(T)|\geq (1-\lambda)\H(T)\, ;\, |V'\wedge\ru(T)|\geq (1-\lambda)\H(T)\}$. We start by finding an upper bound of $\P(E)$ in terms of $\P(A\cap E)$.
We condition on $T$, we recall the definition of $\rL_\lambda(T)$ from (\ref{def_rL_lambda}), and we use $\delta n_0\geq 2$ to obtain
\[\P(A\cap E)=\E{\un_E\frac{\rL_\lambda(T)(\rL_\lambda(T)-1)}{ n_0( n_0-1)}}\geq \delta\frac{\delta n_0-1}{ n_0-1}\P(E)\geq \frac{\delta^2}{2}\P(E).\]
We will now bound $\P(A\cap E)$ from above. Observe that the fact that $V\wedge\ru(T)$ and $V'\wedge\ru(T)$ are both ancestors of $\ru(T)$ yields that $V\wedge V'\wedge\ru(T)$ must be equal to one of the two. In particular, it holds that $|V\wedge V'|\geq \min(|V\wedge\ru(T)|,|V'\wedge \ru(T)|)$. Together with $\H(T)\geq |V|$, this implies that $\P(A\cap E)\leq \P\big(|V|\geq (1-\lambda)/\lambda\, ;\, |V\wedge V'|\geq (1-\lambda)|V|\big)$. Thus, since $\lambda<1/2$, it is sufficient to show that
\begin{equation} 
\label{half_leaves_low_alt}
\P\big(2\lambda|V|\geq 1\, ;\, |V\wedge V'|\geq (1-\lambda)|V|\big)\leq 6\lambda.
\end{equation}

 The next step of the proof is to construct the cyclic shifts envisioned at the start of the proof. Let $t$ be a tree, let $v$ be a leaf of $t$, and let $i\in\{0,\ldots, |v|-1\}$. We can uniquely write $v=v_i*w$ where $|v_i|=i$. We construct a map $\psi:t\to\U$ as follows: for any $x\in t$,
\begin{compactenum}
    \item[$(i)$] we set $\psi(v_i)=\varnothing$ and $\psi(v)=w*v_i$;
    \item[$(ii)$] if $x\in t$ is not a descendant of $v_i$, then we set $\psi(x)=w*x$;
    \item[$(iii)$] if $x\in t$ is a descendant of $v_i$ unequal to $v$, then we set $\psi(x)=y$ where $x=v_i*y$.
\end{compactenum}
Then we set $\hat{t}=\psi(t)$. See also Figure~\ref{fig:cyclic}. It is easy to check that $\hat{t}$ is a tree, that $\psi$ is a bijection from $t$ to $\hat{t}$ and that $\rk_{\psi(x)}(\hat{t})=\rk_x(t)$ for all $x\in t$ (in particular $\psi(x)$ is a leaf of $\hat{t}$ if and only if $x$ is a leaf of $t$). It also follows that $\hat{t}$ has the same type as $t$ and so that $\P(T=\hat{t})=\P(T=t)$. Furthermore, we may check case by  that for any $x\in t$,
\begin{equation}
\label{cyclic_shift_spine}
\psi(v)\wedge\psi(x)=\psi(v\wedge x)\quad\text{ and }\quad
|\psi(v\wedge x)|=  \begin{cases}
                        |v|&\text{ if }x=v,\\
                        |v\wedge x|-i&\text{ if }i\leq |v\wedge x|<|v|,\\
                        |v\wedge x|-i+|v|&\text{ if }|v\wedge x|<i.
                    \end{cases}
\end{equation}
Finally, if $v'$ is a leaf of $t$ then we set $\varphi_i(t,v,v')=(\hat{t},\psi(v),\psi(v'))$. More generally, for any $i\in\Z$, we set $\varphi_i(t,v,v')=\varphi_{i_0}(t,v,v')$ where $i_0$ is the unique integer in $\{0,\ldots,|v|-1\}$ such that $i\in i_0+|v|\Z$. This defines a family of mappings $\varphi_i,i\in\Z$ of the set of trees (of type $\nseq$) equipped with two leaves. Then $(i)$-$(iii)$ allow us to identify that
\begin{equation}
\label{inversion_cyclic_shift_spine}
\varphi_{-i}\circ\varphi_i(t,v,v')=(t,v,v').
\end{equation}

We are now ready to conclude the proof by showing (\ref{half_leaves_low_alt}). Recall that conditionally given $T$, vertices $V$ and $V'$ are drawn from its set of leaves without replacement. Let $\Lambda$ be uniformly distributed in $[0,1]$ and independent of $(T,V,V')$. Let us set $(\widehat{T},\widehat{V},\widehat{V}')=\varphi_{\lfloor \Lambda|V|\rfloor}(T,V,V')$. We claim that $(\widehat{T},\widehat{V},\widehat{V}')$ has the same law as $(T,V,V')$: indeed, if $(t,v,v')$ is a tree of type $\nseq$ equipped with two leaves then we deduce from (\ref{inversion_cyclic_shift_spine}) that
\[\P\big((\widehat{T},\widehat{V},\widehat{V'})=(t,v,v')\big)=\P\big((T,V,V')=\varphi_{-\lfloor\Lambda|v|\rfloor}(t,v,v')\big)=\frac{1}{\#\cT(\nseq)}\times\frac{1}{ n_0( n_0-1)}.\]
Note that the second equality follows from the invariance by $\varphi_i,i\in\Z$ of the type and from the independence of $\Lambda$ with $(T,V,V')$. Therefore,
\[\xi:=\P\big(2\lambda |V|\geq 1\, ;\, |V\wedge V'|\geq (1-\lambda)|V|\big)=\P\big(2\lambda|\widehat{V}|\geq 1\, ;\, |\widehat{V}\wedge \widehat{V}'|\geq (1-\lambda)|\widehat{V}|\big).\]
Then, we separate the cases by relying on (\ref{cyclic_shift_spine}) to get the inequality
\begin{multline*}
\xi\leq \P\big(2\lambda |V|\geq 1\, ;\, |V\wedge V'|<\lfloor\Lambda|V|\rfloor\leq |V\wedge V'|+\lambda|V|\big)\\
+\P\big(2\lambda |V|\geq 1\, ;\, \lfloor\Lambda|V|\rfloor+(1-\lambda)|V|\leq |V\wedge V'|\leq |V|\big).
\end{multline*}
Since $\Lambda$ is uniform on $[0,1]$ and $\lambda\leq 1/2$, we can write
\[\xi\leq 2\, \P\big( 2\lambda|V|\geq 1\, ;\, \Lambda\leq \lambda+1/|V|\big)\leq 6\lambda,\]
which ends the proof of (\ref{half_leaves_low_alt}) and thus of the lemma.
\end{proof}

\begin{proof}[Proof of Proposition~\ref{prop:not_line}]
Let $\varepsilon >0$ such that $\varepsilon^{1/3} n_0\geq 2$. Recalling $\rL_{5\varepsilon}(t)$ from (\ref{def_rL_lambda}) and that $\rL_1(t)$ is the number of leaves of $t$, we define a set of bad trees
\[\mathcal{B}_{\varepsilon}=\{t\in\cT(\nseq)\, :\, \H^{(2)}(t)< \varepsilon\H(t),\H(t)\geq 1/\varepsilon,\rL_{5\varepsilon}(t)\leq \varepsilon^{1/3}\rL_1(t)\}\]
that have many leaves pendant off the lower portion of the spine yet are line-like. In this proof, we can assume that $\varepsilon<1/10$ because otherwise, $64\varepsilon^{1/3}\geq 1$ and the desired inequality follows readily. Together with the assumption that $\varepsilon^{1/3} n_0\geq 2$, this allows us to apply Lemma~\ref{lem:half_leaves_low} with $\delta=\varepsilon^{1/3}$ and $\lambda=5\varepsilon$ to find that
\begin{equation} 
\label{step0_not-line}
\P\big(\varepsilon\H(T)\geq 1\, ;\, \varepsilon\H(T)> \H^{(2)}(T)\big)\leq 60\varepsilon^{1/3}+\P\big(T\in\mathcal{B}_{\varepsilon}\big).
\end{equation}
The rest of the proof is devoted to bounding $\P(T\in\mathcal{B}_{\varepsilon})$ from above. As discussed earlier, we will rely on the transformation that moves the subtree stemming from some vertex of a tree to the position of a leaf. This transformation is formally constructed as follows.

Let $t$ be a tree, let $u\in t$, and let $v$ be a leaf of $t$. Note that $v$ is not an ancestor of $u$ because $v$ is a leaf of $t$. If $u$ is not an ancestor of $v$ then we can check that the following conditions characterise a unique tree $\tilde{t}$:
\begin{compactenum}
    \item[$(a)$] $u$ is a leaf of $\tilde{t}$;
    \item[$(b)$] $v\in \tilde{t}$ and $\theta_v\tilde{t}=\theta_u t$, i.e.~for all $w\in\U$, we have $v*w\in\tilde{t}$ if and only if $u*w\in t$;
    \item[$(c)$] for all $w\in\U$ that is not a descendant of $u$ or $v$, we have $w\in \tilde{t}$ if and only if $w\in t$.
\end{compactenum}
The tree $\tilde{t}$ has the same type as $t$. Indeed, we observe that $\rk_u(\tilde{t})=\rk_v(t)=0$; that if $v*w\in\tilde{t}$ then $\rk_{v*w}(\tilde{t})=\rk_{u*w}(t)$; and that if $w\in\tilde{t}$ is not a descendant of $u$ or $v$ then $\rk_w(\tilde{t})=\rk_w(t)$.  Therefore, writing $\varphi(t,u,v)=(\tilde{t},v,u)$ defines a mapping $\varphi$ on the set of trees (of a given type) with a marked vertex and leaf such that the former is not an ancestor of the latter. It is clear that $\varphi$ is an involution, i.e.~that
\begin{equation}
\label{inversion_switch_subtree}
\varphi\circ\varphi(t,u,v)=(t,u,v).
\end{equation}

We also set $\phi(t,u,v)=\tilde{t}$. Recall that $\ru(t)$ stands for the $\leq$-minimal vertex of $t$ such that $|\ru(t)|=\H(t)$. We are going to use $\phi$ to move the subtree stemming from some ancestor of $\ru(t)$. More precisely, we denote by $\ru_\varepsilon(t)$ the unique ancestor of $\ru(t)$ such that $(1-\varepsilon)\H(t)\leq |\ru_\varepsilon(t)|<(1-\varepsilon)\H(t)+1$. Observe that if $\ru_\varepsilon(t)\preceq v$ then $|v\wedge\ru(t)|\geq (1-\varepsilon)\H(t)$. This ensures that the map $f_\varepsilon:(t,v)\in\mathcal{B}_{\varepsilon}^\bullet\longmapsto \phi(t,\ru_\varepsilon(t),v)\in\cT$ is well-defined on the set 
\[\mathcal{B}^\bullet_{\varepsilon}=\{(t,v)\in\cT\times\mathbb{U}\, :\, t\in\mathcal{B}_{\varepsilon}, v\text{ leaf of } t,|v\wedge \ru(t)|<(1-5\varepsilon)\H(t)\}.\]
By definition of $\mathcal{B}_{\varepsilon}$ and $\mathcal{B}_{\varepsilon}^\bullet$ and since $2\varepsilon^{1/3}\leq 1$, we get that
\[\P\big(T\in\mathcal{B}_{\varepsilon}\big) n_0\leq \sum_{t\in\mathcal{B}_{\varepsilon}}2( n_0-\rL_{5\varepsilon}(t))\P(T=t)=2\sum_{t\in\mathcal{B}_{\varepsilon}}\P(T=t)\#\{v\in t\, :\, (t,v)\in\mathcal{B}_{\varepsilon}^\bullet\}.\]
As the law of $T$ is uniform, the fact that $f_\varepsilon(t,v)$ has always the same type as $t$ yields that
\begin{equation}
\label{step_proba_line-like}
\P\big(T\in\mathcal{B}_{\varepsilon}\big) n_0\leq 2\sum_{(t,v)\in\mathcal{B}_{\varepsilon}^\bullet}\P\big(T=f_\varepsilon(t,v)\big)=2\sum_{\tilde{t}\in \cT}\P(T=\tilde{t})\big| f_\varepsilon^{-1}(\{\tilde{t}\})\big|.
\end{equation}

Now we need to bound $\big|f_\varepsilon^{-1}(\{\tilde{t}\})\big|$ for any $\tilde{t}\in f_\varepsilon(\mathcal{B}_{\varepsilon}^\bullet)$. To do this, it will be useful to define $\ru^{(2)}(\tilde{t})$ as the $\leq$-minimal vertex of $\tilde{t}$ such that $|\ru^{(2)}(\tilde{t})|-|\ru^{(2)}(\tilde{t})\wedge\ru(\tilde{t})|=\H^{(2)}(\tilde{t})$. Note that this vertex is well-defined thanks to (\ref{second-height_alt1}). Furthermore, we denote by $A_\varepsilon(\tilde{t})$ the set of pairs $(\tilde{v},\tilde{u})\in\tilde{t}\times\tilde{t}$ such that $\tilde{u}$ is a leaf of $\tilde{t}$ with $|\tilde{u}|\geq (1-\varepsilon)\H(\tilde{t})$ and such that $\tilde{v}$ is an ancestor of $\ru^{(2)}(\tilde{t})$ with $0\leq \frac{\varepsilon}{1-\varepsilon}|\tilde{u}|-|\ru^{(2)}(\tilde{t})|+|\tilde{v}|<2$.

We postpone the proof of the following assertion:
\begin{equation}
\label{assert_switch_subtree}
\forall (t,v)\in\mathcal{B}_{\varepsilon}^\bullet,\quad \textit{if }\quad f_\varepsilon(t,v)=\tilde{t}\quad\textit{ then }\quad (v,\ru_\varepsilon(t))\in A_\varepsilon(\tilde{t})\textit{ and }\big|A_\varepsilon(\tilde{t})\big|\leq 2\rL_{5\varepsilon}(t).
\end{equation}
By construction, we always have $\varphi(t,\ru_\varepsilon(t),v)=(f_\varepsilon(t,v),v,\ru_\varepsilon(t))$ for any $(t,v)\in\mathcal{B}_{\varepsilon}^\bullet$. Since we know from (\ref{inversion_switch_subtree}) that $\varphi$ is an involution, we get that the size of $f_\varepsilon^{-1}(\{\tilde{t}\})$ is bounded by that of $A_\varepsilon(\tilde{t})$. If in addition the type of $\tilde{t}$ is $\nseq$ then the same holds for $t$ which thus has $ n_0$ leaves. Therefore, combining (\ref{step_proba_line-like}), (\ref{assert_switch_subtree}), and the definition of $\mathcal{B}_{\varepsilon}^\bullet$ entails that
\[\P\big(T\in\mathcal{B}_{\varepsilon}\big)\leq \frac{2}{ n_0}\sum_{\tilde{t}\in\cT}2\varepsilon^{1/3} n_0\P(T=\tilde{t})=4\varepsilon^{1/3}.\]
We deduce from (\ref{step0_not-line}) that 
$\P\big(\varepsilon\H(T)\geq 1 ; \varepsilon\H(T)>\H^{(2)}(T)\big)\leq 64\varepsilon^{1/3}$ as desired.
\end{proof}

\begin{proof}[Proof of (\ref{assert_switch_subtree})]
Here, let $(a)$-$(c)$ refer to the conditions used to construct $\varphi$ in the previous proof. We recall that $\ru(t)$ is the $\leq$-minimal vertex of $t$ with maximum height, and that $\ru_\varepsilon(t)$ is the ancestor of $\ru(t)$ verifying $(1-\varepsilon)\H(t)\leq |\ru_\varepsilon(t)|<(1-\varepsilon)\H(t)+1$. First note that $v\in\tilde{t}$ and that $\ru_\varepsilon(t)$ is a leaf of $\tilde{t}$ by construction. We will consider the word $w_\varepsilon$ characterised by $\ru(t)=\ru_\varepsilon(t)*w_\varepsilon$. We readily get $|w_\varepsilon|=\H(t)-|\ru_\varepsilon(t)|$. Plus, thanks to condition $(b)$, it is easy to check that $w_\varepsilon=\ru(\theta_{\ru_\varepsilon(t)}t)=\ru(\theta_v\tilde{t})$.

The first step of our proof is to compare the heights of $t$ and $\tilde{t}$ and to locate $\ru(\tilde{t})$. The maximum of $|x|$ where $x\in \tilde{t}$ ranges among the descendants of $v$ is equal to $|v|+\H(\theta_v\tilde{t})=|v|+|\ru(\theta_v\tilde{t})|=|v|+|w_\varepsilon|=|v|+\H(t)-|\ru_\varepsilon(t)|\leq |v|+\varepsilon\H(t)$. Plus, (\ref{second-height_alt1}) gives us $|v|-|v\wedge\ru(t)|\leq \H^{(2)}(t)$. Then, we get $|v\wedge\ru(t)|<(1-5\varepsilon)\H(t)$ and $|v|-|v\wedge\ru(t)|<\varepsilon\H(t)$ because $(t,v)\in\mathcal{B}_{\varepsilon}^\bullet$. Thus, we find that
\begin{equation}
\label{new_branch_doesnt_reach_height}
\max_{x\in\tilde{t},v\preceq x}|x|<(1-3\varepsilon)\H(t).
\end{equation}
However, we also have $\H(\tilde{t})\geq |\ru_\varepsilon(t)|\geq (1-\varepsilon)\H(t)$. Therefore, the vertex $\ru(\tilde{t})$ is not a descendant of $v$ and so $\ru(\tilde{t})\in t$ by $(a)$ and $(c)$. It follows that 
\begin{equation}
\label{height_comparison_after_switch}
(1-\varepsilon)\H(t)\leq \H(\tilde{t})\leq\H(t).
\end{equation}
In particular, we readily obtain $|\ru_\varepsilon(t)|\geq (1-\varepsilon)\H(\tilde{t})$ as desired.

The second step of our proof is to identify that $\ru^{(2)}(\tilde{t})=v*w_\varepsilon$. On the one hand, recall that $\ru(\tilde{t})\in t$ and that $|x|\leq |\ru(\tilde{t})|$ for any $x\in\tilde{t}$. If $x\in\tilde{t}$ is not a descendant of $v$ then $x\in t$ by $(a)$ and $(c)$, which thus entails that $|x|-|x\wedge\ru(\tilde{t})|=\min(|x|,|\ru(\tilde{t})|)-|x\wedge\ru(\tilde{t})|\leq \H^{(2)}(t)$. On the other hand, $y\wedge\ru(\tilde{t})$ cannot be a descendant of $v$ for any $y\in \tilde{t}$ because $\ru(\tilde{t})$ is not a descendant of $v$ in the first place. Hence, if $v\preceq y$ then $x\wedge\ru(\tilde{t})\prec v$ and so $x\wedge\ru(\tilde{t})=v\wedge\ru(\tilde{t})$. An application of (\ref{second-height_alt1}) then yields to
\[\H^{(2)}(\tilde{t})\geq \H(\theta_v\tilde{t})+|v|-|v\wedge\ru(\tilde{t})|\geq |w_\varepsilon|+1=\H(t)-|\ru_\varepsilon(t)|+1>\varepsilon\H(t).\]
Recall that $\H^{(2)}(t)<\varepsilon\H(t)$ as $t\in\mathcal{B}_{\varepsilon}$. It follows from the two previous arguments that $v\preceq \ru^{(2)}(\tilde{t})$, which further implies that $\ru^{(2)}(\tilde{t})=v*\ru(\theta_v\tilde{t})=v*w_\varepsilon$.

Therefore, we have shown that $v\preceq \ru^{(2)}(\tilde{t})$ as desired. Plus, it holds that $|v|-|\ru^{(2)}(\tilde{t})|=-|w_\varepsilon|=|\ru_\varepsilon(t)|-\H(t)$. Then, the definition of $\ru_\varepsilon(t)$ yields $0\leq \frac{\varepsilon}{1-\varepsilon}|\ru_\varepsilon(t)|-|\ru^{(2)}(\tilde{t})|+|v|<\frac{1}{1-\varepsilon}$. The desired inequality follows from the assumption that $\varepsilon<1/10$.

Now, let $(\tilde{u},\tilde{v})\in \tilde{t}\times\tilde{t}$ be arbitrary. If $(\tilde{u},\tilde{v})\in A_\varepsilon(\tilde{t})$ then the integer $|\tilde{v}|$ must belong to an interval determined by $(\tilde{t},\tilde{u})$ whose length is strictly smaller then $2$. In other words, given $\tilde{t}$ and $\tilde{u}$, if $(\tilde{u},\tilde{v})\in A_\varepsilon(\tilde{t})$ then $|\tilde{v}|$ can take at most two values. As the ancestors of $\ru^{(2)}(\tilde{t})$ are uniquely determined by their heights, we deduce that \[\big|A_\varepsilon(\tilde{t})\big|\leq 2\big|\{\tilde{u}\text{ leaf of } \tilde{t}\, :\, |\tilde{u}|\geq (1-\varepsilon)\H(\tilde{t})\}\big|.\]
Next, we assume that $\tilde{u}$ is a leaf of $\tilde{t}$ with $|\tilde{u}|\geq (1-\varepsilon)\H(\tilde{t})$. We have $|\tilde{u}|\geq (1-2\varepsilon)\H(t)$ thanks to (\ref{height_comparison_after_switch}) and then (\ref{new_branch_doesnt_reach_height}) ensures that $\tilde{u}$ is not a descendant of $v$. By $(a)$ and $(c)$, it holds that either $\tilde{u}=\ru_\varepsilon(t)$ or $\tilde{u}$ is a leaf of $t$ distinct from $\ru(t)$. In both cases, $\tilde{u}\in t$ so $(1-2\varepsilon)\H(t)\leq |\tilde{u}|-|\tilde{u}\wedge\ru(t)|+|\tilde{u}\wedge\ru(t)|\leq \H^{(2)}(t)+|\tilde{u}\wedge\ru(t)|$ by (\ref{second-height_alt1}). The fact that $t\in\mathcal{B}_{\varepsilon}$ then entails that $|u\wedge\ru(t)|\geq (1-3\varepsilon\H(t))$. Thus, the number of leaves of $\tilde{t}$ with height at least $(1-\varepsilon)\H(\tilde{t})$ is bounded by $\rL_{3\varepsilon}(t)\leq \rL_{5\varepsilon}(t)$ and so $\big|A_\varepsilon(\tilde{t})\big|\leq 2\rL_{5\varepsilon}(t)$.
\end{proof}

\section{Coding trees by walks}
\label{sec:coding}

\subsection{Depth and breadth-first walks}
\label{sec:depth_breadth}
\begin{figure}
\centering
\begin{subfigure}{0.3\textwidth}
\centering
\includegraphics[page=1,width=\textwidth]{coding_trees}
\end{subfigure}
\hspace{1em}
\begin{subfigure}{0.5\textwidth}
\centering
\includegraphics[page=2,width=\textwidth]{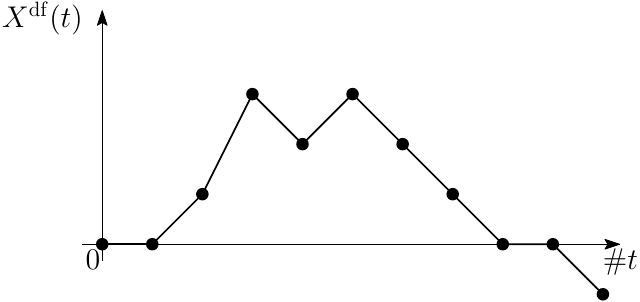}\\
\vspace{1em}
\includegraphics[page=3,width=\textwidth]{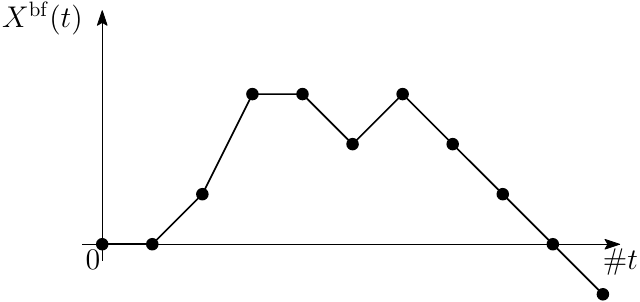}
\end{subfigure}
\caption{A tree $t$ together with its depth-first walk $X^\dfs(t)$ and breadth-first walk $X^\bfs(t)$. The depth-first exploration of $t$ is given by 
$u^\dfs(t)=(\varnothing,1, 11, 111, 112,1121,1122,113,12,121)$ and its breadth-first exploration is given by $u^\dfs(t)=(\varnothing,1, 11,12,111,112,113,121,1121,1122)$. }
\end{figure}

Given a tree $t$ of size $n$, we list its vertices in the lexicographic order as $\varnothing=u_0^{\dfs}(t)<u_1^\dfs(t)<\ldots<u_{n-1}^\dfs(t)$. We call the sequence $u^\dfs(t)=(u_i^\dfs(t)\, ;\, 0\leq i\leq n-1)$ the \emph{depth-first exploration} of $t$. Alternatively, we consider the \emph{breadth-first exploration} of $t$, which is the unique ordering $u^\bfs(t)=(u_i^\bfs(t)\, ;\, 0\leq i\leq n-1)$ of all the elements of $t$ that satisfies the following property: for all $0\leq i<j\leq n-1$, either $|u_i^\bfs(t)|<|u^\bfs_j(t)|$, or both $|u_i^\bfs(t)|=|u_j^\bfs(t)|$ and $u_i^\bfs(t)<u_j^\bfs(t)$ hold. Informally, the breadth-first exploration successively traverses every level of $t$ in the lexicographic order. 

Recall that if $u\in t$ then $\rk_u(t)$ is the number of children of $u$ in $t$. Then we set $X_0^\dfs(t)=X_0^\bfs(t)=0$ and for all $0\leq i\leq n-1$,
\[X_{i+1}^\dfs(t)=X_i^\dfs(t)+\rk_{u_i^\dfs(t)}(t)-1\quad\text{ and }\quad X_{i+1}^\bfs(t)=X_i^\bfs(t)+\rk_{u_i^\bfs(t)}(t)-1.\]
The integer-valued sequences $X^\dfs(t)=(X_i^\dfs(t)\, ;\, 0\leq i\leq n)$ and $X^\bfs(t)=(X_i^\bfs(t)\, ;\, 0\leq i\leq n)$ are called the \emph{depth-first walk} and the \emph{breadth-first walk} of $t$ respectively. We point out that $X^\dfs(t)$ is also commonly known as \emph{the \L ukasiewicz path} of $t$ in the literature. These two sequences characterise their associated tree and highlight some of its properties. In this work, we will use the following folklore result, which involves the set of finite non-negative excursions of downward skip-free walks defined as follows:
\begin{equation}
\label{set_walks_exc}
\mathbb{W}_{\mathrm{exc}}=\bigcup_{n\geq 1}\big\{(x_j)_{0\leq j\leq n}\in\Z^{n+1}\, :\, x_j\geq 0=x_0, x_n=-1,x_{j}-x_{j-1}\geq -1\text{ for all }j\in[n]\big\}.
\end{equation}

\begin{prop}\label{prop:bf_df_props}
Recall from Section~\ref{sec:construction_around_spine} the right-spinal weights $\rS_u(t)$ for $u\in t$. The depth-first and the breadth-first walks enjoy the following properties.
\begin{enumerate}
    \item[$(i)$] The map $t\longmapsto X^\dfs(t)$ is a bijection from the set $\cT$ of all trees to $\mathbb{W}_{\mathrm{exc}}$.
    \item[$(ii)$] The map $t\longmapsto X^\bfs(t)$ is a bijection from the set $\cT$ of all trees to $\mathbb{W}_{\mathrm{exc}}$.
    \item[$(iii)$] For any tree $t$ of size $n$, we have $X_i^\dfs(t)=\rS_{u_i^{\dfs}(t)}^{\rd}(t)$ for all $0\leq i\leq n$. 
    \item[$(iv)$] For any tree $t$, it holds that $\W(t)-1\leq \max X^\bfs(t)\leq 2\W(t)-1$.
\end{enumerate}
\end{prop}

\begin{proof}
The points $(i)$ and $(ii)$ are well-known, see e.g.~\cite[Proposition 1.1]{legall_trees} or \cite[Lemma 6.3]{pitman2006combinatorial}. One can observe that the value of the depth-first (resp.~breadth-first) walk at time $i$ counts the number of neighbours of the subset of vertices visited by the depth-first (resp.~breadth-first) exploration up to time $i-1$ and distinct from $u_i^\dfs(t)$ (resp.~$u_i^\bfs(t)$). More explicitly, we easily check by induction that for all $0\leq i\leq n$, it holds that
\begin{align}
X_i^\dfs(t)&=\#\{u\in t\, :\, u=u_j^\dfs(t)\text{ and }\overleftarrow{u}=u_l^\dfs(t)\text{ for some }1\leq l< i<j\leq n-1\},\label{eq:enum_dfs}\\
X_i^\bfs(t)&=\#\{u\in t\, :\, u=u_j^\bfs(t)\text{ and }\overleftarrow{u}=u_l^\bfs(t)\text{ for some }1\leq l< i<j\leq n-1\}.\label{eq:enum_bfs}
\end{align}
This implies $(iii)$ by definitions of the lexicographic and genealogical order. It also follows that if $u^\bfs_i(t)$ is the $\leq$-minimal vertex of the $h$-th level then $X_i^\bfs(t)+1=\#\{v\in t\, :\, |v|=h\}$. We then readily get the first inequality in $(iv)$. For the second inequality, note that if $u^\bfs_i(t)$ is contained in the $h$-th level, then the set of vertices enumerated by $X_i^\bfs(t)$ in \eqref{eq:enum_bfs} is contained in the $h$-th and $(h+1)$-th level and excludes $u^\bfs_i(t)$ itself. Both of these levels contain at most $\W(t)$ vertices, so  the second inequality in $(iv)$ follows. 
\end{proof}

\subsection{Discrete exchangeable bridges and Vervaat transform}
\label{subsec:vervaat}
A finite sequence of real numbers $\mathtt{b}=(b_1,\ldots, b_n)$ with $n\geq 2$ is called a \emph{jump sequence of length $n$} when $\sigma^2(\mathtt{b}):=\sum_{i=1}^n b_i^2>0$. A random sequence $Y=(Y_i\, ,\, 0\leq i\leq n)$ is an \emph{exchangeable bridge with jump sequence $\mathtt{b}$} when, for a uniformly random permutation $\pi$ of $[n]$,
\begin{equation}
\label{exc_bridge_def}
(Y_i,0\le i \le n) \eqdist \left(\sum_{j=1}^i b_{\pi(j)}, 0\le i \le n \right) =\left(\sum_{j=1}^n b_j\I{\pi(j)\leq i}, 0\le i \le n \right).
\end{equation}
This name is justified because $Y_0=0$ and $Y_n=\sum_{i=1}^n b_i$ almost surely, and by the fact that the increments of $Y$ are exchangeable since $\pi$ is a uniform permutation. Another way to describe the law of $Y$ is to say that it is the summation process associated with a total simple sampling of $\mathtt{b}$. 

Let $\nseq=( n_j)_{j\geq 0}$ be a type sequence with $\sum_{j\ge 0} n_j=n$. Let $\mathtt{d}=(d_1,\ldots,d_n)$ satisfy that $ n_j=\#\{i\in [n]\, :\, d_i=j\}$ so that $\sum_{i=1}^n d_i=n-1$. We call $\mathtt{d}$ \emph{a degree sequence corresponding to $\nseq$}. Let $\mathtt{b}=(b_1,\dots,b_n)=(d_1-1,\dots,d_n-1)$ so that $\mathtt{b}$ is a jump sequence with $\sum_{i=1}^n b_i=-1$. We call $\mathtt{b}$ \emph{a jump sequence corresponding to $\nseq$}.

We let $\cS_n$ be the set of permutations of $[n]$ and define
\begin{align*}
\mathbb{W}_{\mathrm{br}}(\nseq)&=\left\{\left(\sum_{j=1}^i b_{\gamma(j)}, 0 \le i \le n \right)\, : \, \gamma\in \cS_n \right\}\text{ and }\\
\mathbb{W}_{\mathrm{exc}}(\nseq)&=\big\{(y_0,\dots,y_n)\in \mathbb{W}_{\mathrm{br}}(\nseq)\, :\, y_0,\dots, y_{n-1}\ge 0\big\}= \mathbb{W}_{\mathrm{br}}(\nseq)\cap \mathbb{W}_{\mathrm{exc}}.
\end{align*}
It is not hard to see that definitions of $\mathbb{W}_{\mathrm{br}}(\nseq)$ and $\mathbb{W}_{\mathrm{exc}}(\nseq)$ do not depend on the choice of a jump sequence $\bseq$ corresponding to $\nseq$.
Then, the next lemma follows from Proposition~\ref{prop:bf_df_props} $(i)$ (resp.\ $(ii)$) and the observation that for any tree $t$ the offspring sizes in $t$ minus $1$ correspond to the increments in $X^\dfs(t)$ (resp.\ $X^\bfs(t)$).
\begin{lem}\label{lem:trees_excursions}
 It holds that $X^{\dfs}$ and $X^\bfs$ are bijections from $\cT(\nseq)$ to $\mathbb{W}_{\mathrm{exc}}(\nseq)$.
\end{lem}

\begin{figure}
\centering
\includegraphics[page=1,scale=0.7]{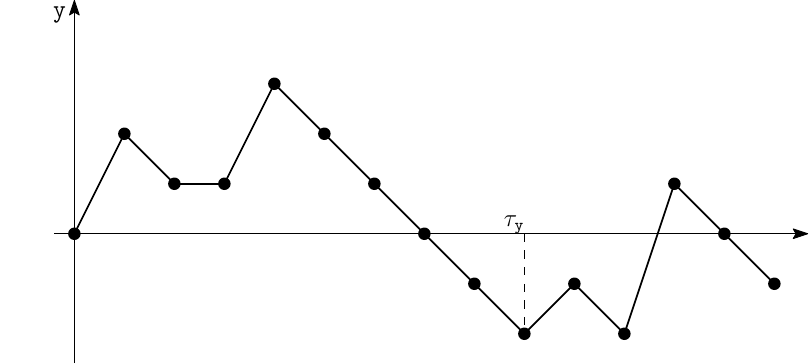}\\
\vspace{1em}
\includegraphics[page=2,scale=0.7]{vervaat}
\caption{\label{fig:vervaat}A bridge $\mathtt{y}$ and its Vervaat transform $V(\mathtt{y})$.}
\end{figure}

For $\mathtt{y}=(y_0,\dots,y_n)\in \mathbb{W}_{\mathrm{br}}(\nseq)$ and $j\in \N$, define 
\[\mathrm{shift}_j(\mathtt{y})=\left(y_{j+i}-y_j, 0\le i \le n \right)\]
where the indices are to be understood modulo $n$, so that $\mathrm{shift}_j(\mathtt{y})\in \mathbb{W}_{\mathrm{br}}(\nseq)$. 
Plus, let 
\begin{equation} 
\label{vervaat_def}\tau_{\mathtt{y}}=\min\{j\in [n]:y_j=\min \mathtt{y} \}\quad\text{ and }\quad V(\mathtt{y})=\mathrm{shift}_{\tau_{\mathtt{y}}}(\mathtt{y}).
\end{equation}
We call $V(\mathtt{y})$ the \emph{Vervaat transform} of $\mathtt{y}$. See Figure~\ref{fig:vervaat}. The following proposition and its continuous counterpart make the Vervaat transform a powerful tool, commonly used across discrete probability theory and stochastic analysis to study excursions of Markov processes \cite{Vervaat}. 
\begin{prop} 
\label{vervvat_bij}
It holds that $\mathtt{y}\mapsto (\tau_{\mathtt{y}},V(\mathtt{y}))$ is a bijection from $\mathbb{W}_{\mathrm{br}}(\nseq)$ to $[n]\times \mathbb{W}_{\mathrm{exc}}(\nseq)$. 
\end{prop}
\begin{proof}
It is straightforward that its inverse is given by $(i,\mathtt{x})\mapsto \mathrm{shift}_{n-i}(\mathtt{x})$, which proves the statement.
\end{proof}

\begin{prop}\label{prop:law_bf_df}
Let $\nseq$ be a type sequence and let $\mathtt{b}=(b_1,\dots, b_n)$ be a jump sequence corresponding to $\nseq$.  Let $Y$ be an exchangeable bridge with jump sequence $\mathtt{b}$ and let $T$ be a uniformly random tree with type $\nseq$. Then, $V(Y)\eqdist X^{\dfs}(T)\eqdist X^{\bfs}(T)$.
\end{prop}
\begin{proof}
Since, by Lemma~\ref{lem:trees_excursions}, $X^{\dfs}$ and $X^{\dfs}$ are bijections from $\cT(\nseq)$ to $\mathbb{W}_{\mathrm{exc}}(\nseq)$, it suffices to show that the Vervaat transform of $Y$ is a uniform element from $\mathbb{W}_{\mathrm{exc}}(\nseq)$. 

Indeed, for $\gamma$ a permutation of $[n]$ and 
\[\mathtt{y}:=\left(\sum_{j=1}^i b_{\gamma(j)}, 0\le i \le n \right)\] observe that 
\[\P(Y=\mathtt{y})=\P(\forall i, b_{\pi(i)}=b_{\gamma(i)})=\P(\forall i, b_{\pi\circ\gamma^{-1}(i)}=b_i)=\P(\forall i,b_{\pi(i)}=b_i)=:C(\mathtt{b}),\]
which does not depend on $\gamma$. Then, for $\mathtt{x}\in \mathbb{W}_{\mathrm{exc}}(\nseq)$, Proposition~\ref{vervvat_bij} yields
\[\P(V(Y)=\mathtt{x})=\sum_{i=1}^{n} \P(Y=\mathrm{shift}_i(\mathtt{x}))=nC(\mathtt{b}),\]
and the claim follows.
\end{proof}

\subsection{Range of exchangeable bridges} For any finite sequence $\mathtt{y}=(y_0,\ldots,y_n)$ of real numbers, we denote by 
\begin{equation} 
\label{range_def}
R(\mathtt{y})=\max_{0\leq i\leq n}y_i\, -\, \min_{0\leq i\leq n}y_i
\end{equation}
its \emph{range}. The goal of this section is to show that for $(\mathtt{b}^n, n\ge 1)$ a procession of jump sequences, where $\mathtt{b}^n$ has length $n$ and sums to $0$, the order of magnitude of the range of exchangeable bridge with jump sequence $\mathtt{b}^n$ is roughly deterministic as $n\to \infty$. More precisely, we want to prove that with high probability, its range is of the same order as the standard deviation $\sigma(\mathtt{b}^n):=\sqrt{\sum_{i=1}^n (b_i^n)^2}$ of its jump sequence $\mathtt{b}$. This has as a consequence that the order of the width of a uniform tree with a given type sequence is deterministic and of the same order as the maximum of its right-spinal weights, and thus will imply Proposition~\ref{width_vs_Luka}.

\begin{thm}
\label{thm:bridge_range}
For all $n\geq 2$, let $\mathtt{b}^n=(b_1^n,\ldots, b_n^n)$ be a jump sequence of length $n$ with $\sum_{i=1}^n b_i^n=0$ and let $Y^n$ be an exchangeable bridge with jump sequence $\mathtt{b}^n$. If $\sigma_n=\sqrt{\sigma^2(\mathtt{b}^n)}$ for all $n\geq 2$, then the sequence of random variables $\sigma_n^{-1}R(Y^n)$ is tight in $(0,\infty)$. 
\end{thm}

This result is a consequence of Hagberg~\cite[Theorem 4]{Hagberg} that asserts the following: \emph{if $|b_1^n|\geq \ldots\geq |b_n^n|$ for all $n\geq 2$ and if $(\sigma_n^{-1} b_i^n)_{n\geq i}$ converges for all $i\geq 1$, then the processes $(\sigma_n^{-1}Y_{\lfloor ns\rfloor}^n\, ;\, s\in[0,1])$ converge in distribution towards an a.s.\ non-constant càdlàg process for the Skorokhod topology on $[0,1]$}. See Billingsley~\cite[Chapter 3]{Billingsley} for information about the Skorokhod topology on $[0,1]$ and to recall that the functions giving the supremum and the infimum on $[0,1]$ are continuous for this topology. Thus, to prove Theorem~\ref{thm:bridge_range}, we may perform a standard diagonal extraction argument to show that the sequence of random variables $\sigma_n^{-1}R(Y^n)$ is relatively compact in distribution and then conclude with Prokhorov's theorem. See also Kallenberg~\cite[Theorems 27.10 and 27.14]{kallenberg2021} for a more general version of Hagberg's limit theorem where the jump sequences may be random.
\medskip

Here, instead, we are going to give an elementary proof of Theorem~\ref{thm:bridge_range} based on the method of moments that also gives us the quantitative control needed to deduce Proposition~\ref{width_vs_Luka}. We fix a jump sequence $\mathtt{b}=(b_1,\ldots,b_n)$ that sums to $0$, a uniformly random permutation $\pi$ on $[n]$, and an exchangeable bridge $Y$ with jump sequence $\mathtt{b}$ such that (\ref{exc_bridge_def}) holds. We write  $\sigma=\sqrt{\sigma^2(\mathtt{b})}$ and we want bound $\sigma^{-1}R(Y)$ from above and below. 

We first treat the upper bound, for which we express the second moment of the bridge $Y$ in terms of the sequence $\mathtt{b}$. Let $1\leq m\leq n$ be an integer. We write $Y_m=\sum_{i=1}^n b_i\I{\pi(i)\leq m}$ and recall that $(\pi(1),\ldots,\pi(n))$ is exchangeable since $\pi$ is a uniform permutation. Developing the squared sum $Y_m^2$ then yields \[\E{Y_m^2}=\P(\pi(1)\leq m)\sum_{i=1}^n b_i^2+\P(\pi(1),\pi(2)\leq m)\sum_{i\ne j} b_i b_j.\]
We see that \begin{equation}\label{eq:mu11}\sum_{i\ne j} b_i b_j=\sum_{i=1}^n b_i\left(-b_i+\sum_{i=1}^n b_i\right )=-\sum_{i=1}^n b_i^2.\end{equation}
Computing the relevant probabilities then entails the identity
\begin{equation}
\label{bridge_second_moment}
\E{Y_m^2}=\frac{m(n-m)}{n(n-1)}\sigma^2.
\end{equation}
We are now ready to prove the upper bound.

\begin{prop}
\label{bridge_range_upper}
If $n\geq 2$ then for all $\varepsilon>0$, it holds that $\P\big(\sigma^{-1}R(Y)\geq \varepsilon^{-1})\leq 12\varepsilon$.
\end{prop}

\begin{proof}
The proof is inspired by that of Etemadi's inequality (see, e.g., Billingsley~\cite[M19]{Billingsley}). Let us write $m=\lceil n/2\rceil$. Note that the events $\{|Y_i|\geq \sigma/\varepsilon\, ;\, \max_{0\leq j<i}|Y_j|<\sigma/\varepsilon\}$ for $0\leq i\leq m$ partition $\{\max_{0\leq i\leq m}|Y_i|\geq \sigma/\varepsilon\}$. It follows that
\begin{align}
\label{upper_bound_bridge1}
\P\left(\max_{0\leq i\leq m}|Y_i|\geq \sigma/\varepsilon\right)&=\sum_{i=0}^m\P\left(\max_{0\leq j<i} |Y_j|<\sigma/\varepsilon\, ;\, |Y_i|\geq \sigma/\varepsilon\right),\\
\label{upper_bound_bridge2}
\E{|Y_m|}\geq \E{|Y_m|\I{\max_{0\leq i\leq m}|Y_i|\geq \sigma/\varepsilon }}&=\sum_{i=0}^m \E{|Y_m|\I{\max_{j<i} |Y_j|<\sigma/\varepsilon\, ;\, |Y_i|\geq \sigma/\varepsilon}}.
\end{align}
Now, let us set $\sgn(x)=1$ for all $x\in[0,\infty)$ and $\sgn(x)=-1$ for all $x\in(-\infty,0)$, so that $|x|=\sgn(x)x$ for all $x\in\R$. We then observe that for any $0\le i\le m$ we have that \[|Y_m|\geq \sgn(Y_i)Y_m=\sgn(Y_i)Y_i+\sgn(Y_i)(Y_m-Y_i)=|Y_i|+\sgn(Y_i)(Y_m-Y_i),\] 
so that (\ref{upper_bound_bridge2}) and the tower principle yield that
\begin{equation}
\label{upper_bound_bridge3}
\E{|Y_m|}\geq \sum_{i=0}^m \E{\I{\max_{j<i} |Y_j|<\sigma/\varepsilon\, ;\, |Y_i|\geq \sigma/\varepsilon}\big(|Y_i|+\sgn(Y_i)\E{Y_m-Y_i\, |\, \pi(1)\,\ldots,\pi(i)}\big)}.
\end{equation}
For any $0\leq i\leq m$, we denote by $\cF_i$ the sigma-field generated by $(\pi(1),\ldots,\pi(i))$. By exchangeability, it holds for any $i+1\leq j\leq n$ that $(\pi(1),\ldots,\pi(i),\pi(j))$ has the same law as $(\pi(1),\ldots,\pi(i),\pi(i+1))$. Therefore, we have
\begin{align*}\E{Y_m-Y_i\, |\, \cF_i}&=\sum_{j=i+1}^m\E{b_{\pi(j)}\, |\, \cF_i}\\
&=(m-i)\E{b_{\pi(i+1)}\, |\, \cF_i}\\
&=\frac{m-i}{n-i}\sum_{j=i+1}^n\E{b_{\pi(j)}\, |\,\cF_i}\\
&=-\frac{m-i}{n-i}Y_i\end{align*}
where in the last line we use that $\sum_{i=1}^n b_{\pi(i)}=0$. We insert this identity into (\ref{upper_bound_bridge3}) and then use (\ref{upper_bound_bridge1}) to obtain
\[\E{|Y_m|}\geq \sum_{i=0}^m\E{\I{\max_{j<i} |Y_j|<\sigma/\varepsilon\, ;\, |Y_i|\geq \sigma/\varepsilon}\frac{n-m}{n-i}|Y_i|}\geq \frac{\sigma(n-m)}{\varepsilon n}\P\big(\max_{0\leq i\leq m}|Y_i|\geq \sigma/\varepsilon\big).\]
Then, we further bound the left-hand side with the Cauchy--Schwarz inequality and (\ref{bridge_second_moment}) to write $\E{|Y_m|}\leq \sqrt{\E{Y_m^2}}\leq \sigma$. Since $m=\lceil n/2\rceil$ and $n\geq 2$, it holds that $(n-m)/n\geq 1/3$. Thus, we have proved that $\P(\max_{0\leq i\leq m}|Y_i|\geq \sigma/\varepsilon)\leq 3\varepsilon$.

Now, let us highlight that the exchangeability implies that $(Y_n-Y_{n-i}\, ;\, 0\leq i\leq n)$ and $Y$ have the same distribution. Then, we recall that $Y_n=0$ and $2m\geq n$ to get 
\begin{align*}\P\big(\max_{0\leq i\leq n}|Y_i|\geq \sigma/\varepsilon\big)&\leq\P\big(\max_{0\leq i\leq m}|Y_i|\geq \sigma/\varepsilon\big)+\P\big(\max_{0\leq i\leq m}|Y_{n-i}|\geq \sigma/\varepsilon\big)\\
&=2\P\big(\max_{0\leq i\leq m}|Y_i|\geq \sigma/\varepsilon\big)\leq 6\varepsilon.\end{align*}
Finally, the deterministic inequality $R(Y)\leq 2\max_{0\leq i\leq n}|Y_i|$ completes the proof.
\end{proof}

Our proof for the lower bound will involve the moments of $\mathtt{b}$. For positive integers $r$ and $\alpha_1,\ldots,\alpha_r$, we adopt the following notation:
\[\mu_{\alpha_1,\ldots,\alpha_r }=\mu_{\alpha_1,\ldots,\alpha_r }(\mathtt{b})=\sum_{\substack{1\leq i_1,\ldots, i_r\leq n\\ \text{distinct}}}\:\:\prod_{j=1}^r b_{i_j}^{\alpha_j}.\]
In particular, the assumption $\sum_{i=1}^n b_i=0$ can be simply written as $\mu_1=0$. Also note that $\mu_2=\sigma^2$. The next lemma identifies that all the moments that will occur in our proof can be expressed in terms of $\mu_2$ and $\mu_4$.

\begin{lem}
\label{moment_sequence}
Under the assumption $\sum_{i=1}^n b_i=0$, the following identities hold:
\[\mu_{1,1}=-\mu_2,\quad \mu_{2,2}=\mu_2^2-\mu_4,\quad \mu_{3,1}=-\mu_4,\quad \mu_{2,1,1}=2\mu_4-\mu_2^2,\quad \mu_{1,1,1,1}=3\mu_2^2-6\mu_4.\]
Moreover, it holds that $0\leq \mu_4\leq \mu_2^2$ which leads to the following inequalities:
\[\mu_{2,2}\leq \mu_2^2,\quad \mu_{3,1}\leq 0,\quad  \mu_{2,1,1}\leq \mu_2^2,\quad  \text{ and }\quad \mu_{1,1,1,1}\leq 3\mu_2^2.\]
\end{lem}

\begin{proof}
The first equality is proved in \eqref{eq:mu11}. For the second equality, we write the sum in the definition of $\mu_{3,1}$ as a double sum and find $\mu_{3,1}=\sum_{i=1}^n b_i^3(\mu_1-b_i)=-\mu_4$. Similarly, for the fourth equality, we see that $\mu_{2,2}=\sum_{i=1}^n b_i^2(\mu_2-b_i^2)=\mu_2^2-\mu_4$. Now, for the fourth equality, we separate the sum indexed by $(i,j,k)$ into two sums, one indexed by $(i,j)$ and the other by $k$, to find
\[\mu_{2,1,1}=\sum_{\substack{1\leq i,j\leq n\\ i\neq j}} b_i^2 b_j(\mu_1-b_i-b_j)=-\mu_{3,1}-\mu_{2,2}=\mu_4+\mu_4-\mu_2^2.\]
The same method gives 
\[\mu_{1,1,1,1}=\sum_{\substack{1\leq i,j,k\leq n\\ \text{distinct}}}b_i b_j b_k(\mu_1-b_i-b_j-b_k).\]
By symmetry, this becomes $\mu_{1,1,1,1}=-3\mu_{2,1,1}=3\mu_2^2-6\mu_4$, which gives the fifth equality of the lemma.

The $b_j^2$ are all non-negative so $b_i^2\leq \mu_2$ for any $1\leq i\leq n$, and then $0\leq \mu_4\leq \mu_2\sum_{i=1}^n b_i^2=\mu_2^2$. The other inequalities readily follow.
\end{proof}

In \eqref{bridge_second_moment} we expressed the second moment of the bridge $Y$ in terms of the moments of the sequence $\mathtt{b}$. We now do the same for the other moments of the bridge that we need in the proof for the lower bound on $\sigma^{-1}R(Y)$. Let $1\leq m\leq n$ be an integer. Recall that $Y_m=\sum_{i=1}^n b_i\I{\pi(i)\leq m}$ and that $(\pi(1),\ldots,\pi(n))$ is exchangeable. In particular, we simply get $\E{Y_m}=\mu_1\P(\pi(1)\leq m)=0$. 
Similarly, developing $(\sum_{i=1}^n b_i\I{\pi(i)\leq m})^4$ gives
\begin{multline*}
\E{Y_m^4}=\mu_4\P(\pi(1)\leq m)+4\mu_{3,1}\P(\pi(1),\pi(2)\leq m)+3\mu_{2,2}\P(\pi(1),\pi(2)\leq m)\\
+6\mu_{2,1,1}\P(\pi(1),\pi(2),\pi(3)\leq m)+\mu_{1,1,1,1}\P(\pi(1),\pi(2),\pi(3),\pi(4)\leq m).
\end{multline*}
Except for the first term, we can bound the relevant probabilities by $\P(\pi(1),\pi(2)\leq m)=\frac{m(m-1)}{n(n-1)}$. In addition, we use the inequalities given by Lemma~\ref{moment_sequence} so that
\begin{equation}
\label{bridge_fourth_moment}
\E{Y_m^4}\leq \frac{m}{n}\mu_4+12\frac{m(m-1)}{n(n-1)}\mu_2^2.
\end{equation}
Now, let us assume that $2m\leq n$. We use the identity $Y_{2m}-Y_m=\sum_{i=1}^n b_i\I{m<\pi(i)\leq 2m}$ to develop $Y_m^2(Y_{2m}-Y_m)^2$ and then get that $\E{Y_m^2(Y_{2m}-Y_m)^2}$ is equal to
\begin{multline*}
\mu_{2,2}\P(\pi(1)\leq m\, ;\, m<\pi(3)\leq 2m)+2\mu_{2,1,1}\P(\pi(1)\leq m\, ;\, m<\pi(3),\pi(4)\leq 2m)\\
+\mu_{1,1,1,1}\P(\pi(1),\pi(2)\leq m\, ;\, m<\pi(3),\pi(4)\leq 2m).
\end{multline*}
Once again, we apply Lemma~\ref{moment_sequence} to obtain a moment inequality:
\begin{equation}
\label{bridge_product_second_moment}
\E{Y_m^2(Y_{2m}-Y_m)^2}\leq \frac{m^2}{n(n-1)}(\mu_2^2-\mu_4)+5\frac{m^2(m-1)}{n(n-1)(n-2)}\mu_2^2.
\end{equation}

We are now ready to bound $\sigma^{-1}R(Y)$ from below.

\begin{prop}
\label{bridge_range_lower}
Let $p\! \geq\! 1$ be an integer. If $n\! \geq\! p^2$ then $\P\big(2\sigma^{-1}R(Y)\leq p^{-1/2})\leq 400 p^{-1}$.
\end{prop}

\begin{proof}
We can assume that $p\geq 4$ because otherwise, $400 p^{-1}\geq 1$ and the desired result becomes obvious. Let us set $m=\lfloor n/p\rfloor$ and $Q=\sum_{i=1}^p(Y_{i\lfloor n/p\rfloor}-Y_{(i-1)\lfloor n/p\rfloor})^2$. We begin by computing the mean and the variance of $Q$. Using the linearity of expectation and the exchangeability of $(Y_1,Y_2-Y_1,\ldots,Y_n-Y_{n-1})$, we obtain
\[\E{Q}=p\E{Y_m^2}\quad\text{ and }\quad \V{Q}=p\E{Y_m^4}+p(p-1)\E{Y_m^2(Y_{2m}-Y_m)^2}-p^2\E{Y_m^2}^2.\]
By inserting the basic inequality $n-p\leq pm\leq n$ into the formula (\ref{bridge_second_moment}), we first find that $n\E{Q}\geq (n-p)(1-1/p)\sigma^2$. Then, the assumptions $n\geq p^2$ and $p\geq 4$ yield that
\begin{equation}
\label{quad_var_mean}
\E{Q}\geq (1-2/p)\sigma^2\geq\tfrac{1}{2}\sigma^2.
\end{equation}
Secondly, we bound the terms involved in the expression of $\V{Q}$ respectively thanks to (\ref{bridge_fourth_moment}), (\ref{bridge_product_second_moment}), and (\ref{quad_var_mean}). Still applying $pm\leq n$, we thereby find that 
\[\V{Q}\leq \mu_4\: + \: 12\tfrac{n-p}{n-1}\mu_2^2/p \: +\: \tfrac{n(p-1)}{(n-1)p}(\mu_2^2-\mu_4) \: +\: 5\tfrac{n-p}{n-2}\tfrac{n(p-1)}{(n-1)p}\mu_2^2/p\: -\: (1-2/p)^2\sigma^4.\]
Recall that $\sigma^2=\mu_2$ and $p\geq 2$, and observe that $n\geq p$ implies that $n(p-1)\leq (n-1)p$. We also apply the elementary inequality $(1-x)^2\geq 1-2x$ with $x=2/p$. Finally, we show that
\begin{equation}
\label{quad_var_variance}
\V{Q}\leq 21\sigma^4/p.
\end{equation}

The rest of the proof relies on the observation that $Q\leq pR(Y)^2$. Thus, if $2\sqrt{p}R(Y)\leq \sigma$ then $Q\leq \sigma^2/4$. In this case, (\ref{quad_var_mean}) further entails that $|Q-\E{Q}|\geq \sigma^2/4$. We conclude thanks to Chebyshev's inequality and (\ref{quad_var_variance}) by writing
\[\P\big(\sigma^{-1}R(Y)\leq \tfrac{1}{2\sqrt{p}}\big)\leq \P\big(|Q-\E{Q}|\geq \sigma^2/4\big)\leq \frac{16\V{Q}}{\sigma^4}\leq 400/p.\qedhere\]
\end{proof}

\begin{proof}[Proof of Theorem~\ref{thm:bridge_range}]
The theorem directly follows from Propositions~\ref{bridge_range_upper} and \ref{bridge_range_lower}.
\end{proof}

\subsection{Comparison of the width and the maximum spinal weight}

Here, we combine the distributional identity given by Proposition~\ref{prop:law_bf_df} with the tails bounds of Propositions~\ref{bridge_range_upper} and \ref{bridge_range_lower} to prove Proposition~\ref{width_vs_Luka}. Recall that $\nseq=(n_i)_{i\geq 0}$ is a type sequence with size $n\geq 3$ and that $T$ is uniformly distributed on $\cT(\nseq)$. Let $\varepsilon>0$ with $\varepsilon^{4/3}\sqrt{n_0-1}\geq 2^6$.

\begin{proof}[Proof of Proposition~\ref{width_vs_Luka}]

We can assume that $\varepsilon\leq 1/4$ because we would have $\sqrt[3]{2^{22}}\varepsilon^{2/3}\geq 1$ otherwise, which is an upper bound for any probability. Now, recall from Proposition~\ref{prop:bf_df_props} $(iii)$ and $(iv)$ that $\max X^{\dfs}(T)=\max_{u\in T}\rS_u^{\rd}(T)$ and that $\W(T)\le  1+ \max X^{\bfs}(T)$. Next, let $\mathtt{b}=(b_1,\dots, b_n)$ be a jump sequence corresponding to $\nseq$, as defined in Section~\ref{subsec:vervaat}, and let $Y=(Y_i,0\leq i\leq n)$ be an exchangeable bridge with jump sequence $\mathtt{b}$, as in (\ref{exc_bridge_def}). We then recall from Proposition~\ref{prop:law_bf_df} that $X^{\bfs}(T)$ and $X^{\dfs}(T)$ are both distributed as $V(Y)$ (the Vervaat transform of $Y$). From the expressions (\ref{vervaat_def}) and (\ref{range_def}), we see that the range is invariant under taking the Vervaat transform, so that $R(Y)\eqdist R(X^{\bfs}(T))=1+\max X^{\bfs}(T)$ and similarly $R(Y)\eqdist 1+\max X^{\dfs}(T)$. Therefore, for any $r\geq 1$, we can write
\begin{equation} 
\label{Wd-vs-MaxX_step1}
\P\big(\varepsilon\W(T)\geq \max X^{\dfs}(T)\big)\leq \P\big(\varepsilon R(Y)\geq r\big)+\P\big(r\geq R(Y)-1\big).
\end{equation}

We now want to use Propositions~\ref{bridge_range_upper} and \ref{bridge_range_lower}, but they only apply to jump sequences that sum to $0$. Thus, we consider $\tilde{\mathtt{b}}=(b_1+\tfrac{1}{n},\dots, b_n+\tfrac{1}{n})$ instead, and we observe that $\tilde{Y}=(Y_i+\tfrac{i}{n},0\leq i\leq n)$ is an exchangeable bridge with jump sequence $\tilde{\mathtt{b}}$. Then, we note that $|R(\tilde{Y})-R(Y)|\leq 1$ and denote by $\tilde{\sigma}=\sqrt{\sum_{i=1}^n (b_i+\tfrac{1}{n})^2}$ the standard deviation of $\tilde{b}$. We obtain from (\ref{Wd-vs-MaxX_step1}) that for any $r\geq 1$,
\[\P\big(\varepsilon\W(T)\geq \max_{u\in T} \rS_u^{\rd}(T)\big)\leq \P\big( \tilde{\sigma}^{-1}R(\tilde{Y})\geq \tilde{\sigma}^{-1}(r/\varepsilon-1)\big)+\P\big(2\tilde{\sigma}^{-1}R(\tilde{Y})\leq 2\tilde{\sigma}^{-1}(r+2)\big).\]

For the rest of the proof, we choose and fix $r=2^{-7/3}\tilde{\sigma}\varepsilon^{1/3}$. By computing that \[\tilde{\sigma}^2=\sum_{i=1}^n b_i^2-\tfrac{1}{n}=\sum_{i\geq 0}(i-1)^2 n_i -\tfrac{1}{n}\geq n_0-1,\]
the assumed inequalities $\varepsilon\leq 1/4$ and $\varepsilon^{4/3}\sqrt{ n_0-1}\geq 2^6$ imply that $r\geq 4\varepsilon$, and so that $r/\varepsilon-1\geq \tfrac{3r}{4\varepsilon}$. Thus, Proposition~\ref{bridge_range_upper} yields that
\begin{equation}
\label{Wd-vs-MaxX_step3}
\P\big( \tilde{\sigma}^{-1}R(\tilde{Y})\geq \tilde{\sigma}^{-1}(r/\varepsilon-1)\big)\leq 12\cdot\tfrac{4\varepsilon}{3r}\tilde{\sigma}=2^{19/3}\varepsilon^{2/3}.
\end{equation}
Still relying on $\tilde{\sigma}\geq\sqrt{ n_0-1}$, we also derive from the assumptions that $r+2\leq \tfrac{32}{30}r$. Next, we write $\delta=\tfrac{32r}{15\tilde\sigma}=\frac{2^{8/3}}{15}\varepsilon^{1/3}$ and set $p=\lfloor \delta^{-2}\rfloor$ to lighten the notation. We have $\delta\leq 1/3$ and then $\delta^{-2}\geq p\geq \tfrac{8}{9}\delta^{-2}$. In particular, the integer $p$ satisfies $p^2\leq 15^4 2^{-32/3}\varepsilon^{-4/3}\leq 2^6\varepsilon^{-4/3}$ which is by assumption at most $\sqrt{ n_0-1}\leq n$. Hence, we can apply Proposition~\ref{bridge_range_lower} to get
\begin{equation}
\label{Wd-vs-MaxX_step4}
\P\big(2\tilde{\sigma}^{-1}R(\tilde{Y})\leq 2\tilde{\sigma}^{-1}(r+2)\big)\leq \P\big(2\tilde{\sigma}^{-1} R(\tilde{Y})\leq \delta\big)\leq 400\cdot \tfrac{9}{8}\delta^2=2^{19/3}\varepsilon^{2/3}.
\end{equation}
Gathering (\ref{Wd-vs-MaxX_step3}) and (\ref{Wd-vs-MaxX_step4}) gives the desired result.
\end{proof}

\begin{remark}
The same arguments can be used to prove the following. For any $n\geq 1$, let $\nseq^n$ be a type sequence size $n$ such that $ n_0\geq 2$, let $\mathtt{b}^n=(b^n_1,\dots, b^n_n)$ be a jump sequence corresponding to $\nseq^n$ and let $\sigma_n=\sqrt{\sum_{i=1}^n (b^n_i)^2}$. If $T_n$ are uniformly random trees with respective type $\nseq^n$, then $\left(\sigma^{-1}_n\W(T_n)\right)_{n\ge 1}$ and $\left(\sigma^{-1}_n\max X^\dfs(T_n)\right)_{n\ge 1}$ are tight in $(0,\infty)$.
\end{remark}

\section{The height of the trees stemming from the spine}
\label{sec:height_spine}

In this section, we show Proposition~\ref{bound_product_right-weight-height}, which is a key ingredient of the proof of Theorem~\ref{thm:nseq}. In this regard, we want to uniformly control the height of subtrees stemming from the right of any spine of a uniformly random tree with fixed type. First, let us introduce some notation to denote these subtrees.

\begin{figure}
\centering
\includegraphics[page=1,scale=0.5]{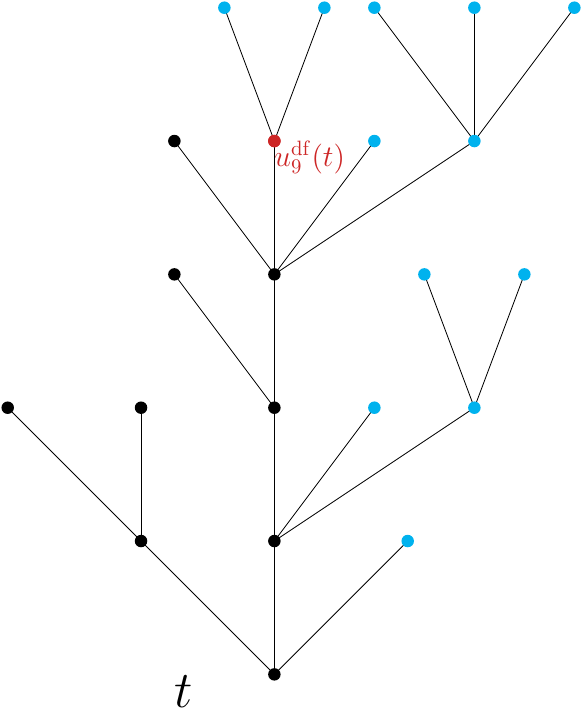}
\hspace{1em}
\includegraphics[page=2,scale=0.5]{spine_trees}
\caption{\label{fig:spine_trees} A depiction of a tree $t$ and the tree $t^{(9)}$ that is obtained by taking all subtrees stemming from $u^{\dfs}_9(t)$ and from younger siblings of  $u^{\dfs}_9(t)$ itself and its ancestors and attaching these trees to a root vertex.}
\end{figure}

Let $n\ge 1$, let $t$ be a tree of size $n$, let $0\le k \le n$. Recall from Section~\ref{sec:depth_breadth} that $X^{\dfs}(t)$ stands for the depth-first walk of $t$. We define $t^{(k)}$ as the tree whose depth-first walk $X^{\dfs}(t^{(k)})$ is $(0,X^\dfs_{k}(t),X^\dfs_{k+1}(t),\dots, X^\dfs_{n}(t))$ so that $t^{(k)}$ has size $1\le n-k+1 \le n+1$. See Figure~\ref{fig:spine_trees}. Then, for $r=X^\dfs_{k}(t)+1$, the root in $t^{(k)}$ has degree $r$ and $\theta_1 t^{(k)}, \dots,\theta_r t^{(k)} $ are the subtrees of $t$ stemming from $u^{\dfs}_k(t)$ and from the younger siblings of $u^{\dfs}_k(t)$ itself and its ancestors, that we denote by $u_k^{\dfs}(t)=v_1< \ldots< v_r$. Thus, $\H(t^{(k)})-1$ is the maximal height of any such subtree. For any $2\leq i\leq r$, if $v\in\theta_i t^{(k)}$ then we see that the most recent common ancestor of $u_k^{\dfs}(t)$ and $v_i*v$ is the parent of $v_i$, so that $1+|v|=|v_i*v|-|u_k^{\dfs}(t)\wedge(v_i*v)|$. Similarly, if $v\in \theta_1t^{(k)}$ then $|v|=|u_k^{\dfs}(t)*v|-|u_k^{\dfs}(t)\wedge (u_k^{\dfs}(t)*v)|$. Therefore, it holds that
\begin{equation}
\label{height_forest_right-spinal}
\H(t^{(k)})=\max_{v\in t, u\leq v}|v|-|u\wedge v|+\I{u\preceq v},\quad\text{ where }u=u_k^{\dfs}(t).
\end{equation}
Most of this section is dedicated to proving the following lemma. 

\begin{lem}\label{lem:height_forest}
Let $\nseq=(n_i)_{i\geq 0}$ be a type sequence of size $n\geq 3$, and let $T$ be a uniformly random tree with type $\nseq$. Then, for any $(x_0,\dots, x_k)$ in the support of $(X^\dfs_i(T), 0\le i \le k)$, it holds that for all $s>0$,
\[\P\big(X^\dfs_k(T)\cdot \H(T^{(k)})>sn\log n\, \big|\, X^\dfs_0(T)=x_0,\dots, X^\dfs_k(T)=x_k\big)\leq 3n^{1-s/2}.\]
\end{lem}

Before we prove the lemma, we use it to show Proposition~\ref{bound_product_right-weight-height}.

\begin{proof}[Proof for Proposition~\ref{bound_product_right-weight-height} using Lemma \ref{lem:height_forest}]
We know from Proposition~\ref{prop:bf_df_props} $(iii)$ that if $u=u_k^{\dfs}(T)$ with $0\leq k\leq n-1$, then $\rS_u^{\rd}(T)=X_k^{\dfs}(T)$. Also recall from (\ref{height_forest_right-spinal}) that $\H(T^{(k)})=\max_{v\in T, u\leq v}|v|-|u\wedge v|+\I{u\preceq v}$ with $u=u_k^{\dfs}(T)$,. It immediately follows that $\H(T^{(k)})\geq\max_{v\in T, u\leq v}|v|-|u\wedge v| $. Therefore, by a union bound over all $u\in T$,
\[\P\Big(\max_{\substack{u,v\in T\\ u\leq v}}(|v|-|u\wedge v|)\rS_u^{\rd}(T)\geq sn\log n\Big)\leq \sum_{k=0}^{n-1} \P\big(X^\dfs_k(T)\cdot \H(T^{(k)})\geq sn\log n\big).\]
For each $0\leq k\leq n-1$, we condition $\P\big(X^\dfs_k(T)\cdot \H(T^{(k)})\geq sn\log n\big)$ on $X_0^{\dfs}(T),\ldots,X_k^{\dfs}(T)$ before applying Lemma~\ref{lem:height_forest} and the tower property to obtain the desired upper bound $n\cdot 3n^{1-s/2}=3n^{2-s/2}$.
\end{proof}

We will now prove Lemma~\ref{lem:height_forest}. Recall that, by Lemma~\ref{lem:trees_excursions}, $X^\dfs$ is a bijection from $\cT(\nseq)$ to $\mathbb{W}_{\mathrm{exc}}(\nseq)$ so that $X^\dfs(T)$ is a uniform pick from $\mathbb{W}_{\mathrm{exc}}(\nseq)$. 
For $i\ge 0$, we set 
\[n'_i=n'_i(x_0,\dots, x_k)= n_i-|\{1\le j \le k: x_j=x_{j-1}+i-1\}|+\I{i=x_k+1}\]
and we write $\nseq'=\nseq'(x_0,\dots,x_k)=(n'_i)_{i\ge 0}$, so that if $t$ has type $\nseq$ and if $X^\dfs_0(t)=x_0,\dots, X^\dfs_k(t)=x_k$ then $t^{(k)}$ has type $\nseq'$. 
We observe that, conditional on the event $\{X^\dfs_0(T)=x_0,\dots, X^\dfs_k(T)=x_k\}$,
it holds that $X^\dfs(T^{(k)})$ is a uniform pick from $\mathbb{W}_{\mathrm{exc}}(\nseq')$ conditional on the event that the first increment equals $x_k$. Thus, conditional on the event $\{X^\dfs_0(T)=x_0,\dots, X^\dfs_k(T)=x_k\}$, it holds that $T^{(k)}$ is a uniform tree with type $\nseq'$ conditional on the event that its root has degree $x_k+1$. Therefore, Lemma~\ref{lem:height_forest} is a consequence of the following proposition.

\begin{prop}\label{prop:forest_short}
Let $n\geq 3$, let $1\leq r\leq n$, and let $\nseq=( n_i)_{i\ge 0}$ be a type sequence with size at most $n+1$ and such that $ n_r\geq 1$. If $T$ is a uniformly random tree with type $\nseq$ conditioned to have $\rk_\varnothing(T)=r$, then for all $s>0$, it holds that
\[ \P(r\H(T)>sn\log n)\leq 3n^{1-s/2}.\]
\end{prop}

This proposition follows from a stochastic domination result. Fix $1 \le r \le n$ and let $\nseq^*=\nseq^*(n,r)=(n^*_i)_{i\ge 0}$ be the type of a tree with size $n+1$ that only contains leaves, vertices with degree $1$ and one additional vertex of degree $r$, i.e.
\[n^*_i= r\I{i=0}+(n-r)\I{i=1}+\I{i = r}\]
for all $i\geq 0$. Then, the following proposition states that across types $\nseq=( n_i)_{i\ge 0}$ of trees with at most $n+1$ vertices and $ n_r\ge 1$, the type $\nseq^*$ maximises the height of a uniformly random tree with type $\nseq$ conditioned to have root degree $r$. 

\begin{prop}\label{prop:stoch_dom}
Fix  $1\le r \le n$ and let $\nseq=( n_i)_{i\ge 0}$ be a type sequence with size at most $n+1$ vertices and such that $ n_r\ge 1$. Then, for $T$ a uniform pick from $\cT(\nseq)$ conditional on $\rk_\varnothing(T)=r$ and $T^*$ a uniform pick from $\cT(\nseq^*)$ conditional on $\rk_\varnothing(T^*)=r$, it holds that $\H(T)$ is stochastically dominated by $\H(T^*)$.
\end{prop}

Then, Proposition~\ref{prop:forest_short} follows from Proposition~\ref{prop:stoch_dom}, that we will soon prove, and the following proposition. 

\begin{prop}\label{prop:star_short}
Let $1\le r\le n$ and let $\nseq^*=\nseq^*(n,r)$ be as above. If $T^*$ is a uniformly random pick from $\cT(\nseq^*)$ conditional on $\rk_\varnothing(T^*)=r$, then for all $s>0$, it holds that
\[ \P( r\H(T^*)> sn\log n)\leq  3n^{1-s/2}.\]
\end{prop}

\begin{proof}
We can assume that $n\ge 2$ and $s\geq 2$ because $3n^{1-s/2}\geq 1$ otherwise. Since $T^*$ has $n+1$ vertices its height is at most $n$,  so for $r=1$ we see that $\p{ \H(T^*)>sn\log n}=0$ for all $n\ge 2$ and $s\ge 2$, and we thus may assume that $r\ge 2$. 

Let $\cT^*$ denote the set of trees $t$ with type $\nseq^*$ and $\rk_\varnothing(t)=r$. Then, for $t\in \cT^*$ the trees $\theta_1 t, \dots, \theta_r t$ are paths of total length $n-r$ and in fact, $t \longmapsto (\H(\theta_1 t),\dots, \H(\theta_r t) )$ is a bijection from $\cT^*$ to the set of all weak compositions of $n-r$ into $r$ parts, which has $\binom{n-1}{r-1}$ elements  by a standard stars-and-bars argument. 
Similarly, for any integers $h\geq 0$ and $1\le k \le r$, we check that $t\longmapsto (\H(\theta_i t)-h\I{i=k},1\leq i\leq r)$ yields a bijection from $\{t\in \cT^*: \H(\theta_k t)\ge h\}$  to the set of all weak compositions of $n-r-h$ into $r$ parts, which has $\binom{n-h-1}{r-1}$ elements. Therefore,
\[\left|\{t\in \cT^*: \H(t)\geq h+1\}\right|\leq r \binom{n-h-1}{r-1}.\]
For any integer $h\geq 0$, we compute
\[\P(\H(T^*)\geq h+1)\leq r \frac{(n-1-h)!(n-r)!}{(n-1)!(n-r -h)!}=r\prod_{i=1}^{h}\left(1-\frac{r-1}{n-i} \right)\leq r\left(1-\frac{r-1}{n}\right)^{h}.\]
Thus, for $2\le r \le n$,
\[\P\big(\H(T^*)>\tfrac{sn}{r}\log n\big)\leq  r \exp\Big(\tfrac{r-1}{n}-s \tfrac{r-1}{r}\log n\Big)\leq 3n^{1-s/2}.\qedhere\]
\end{proof}

Proposition~\ref{prop:stoch_dom} is a consequence of a more general result. To be precise, there is a partial order $\preceq$ on type sequences that satisfies that, for $\nseq$ and $\nseq^*$ as in Proposition~\ref{prop:stoch_dom}, $\nseq \preceq \nseq^*$. This partial order induces a stochastic ordering on the height of uniformly random trees with a given type and fixed root degree. This is the content of the next theorem.

\begin{thm}\label{thm:stoch_ord}
Fix $r\ge 1$ and let $\preceq$ be the partial order on type sequences defined by the following covering relation: for types $\nseq'=(n'_i)_{i\ge 0}$ and $\nseq=( n_i)_{i\ge 0}$, say $\nseq'=(n'_i)_{i\ge 0}$ covers $\nseq=( n_i)_{i\ge 0}$ if either of the following two statements holds:
\smallskip
\begin{compactenum}
    \item[$(a)$] It holds that $n'_1= n_1+1$ and $n'_i= n_i$ for $i\ne 1$, or
    \item[$(b)$] There exist $k, \ell $ with $1\le k\le \ell$ such that 
    \[  n_i = n'_i -\I{i=k} +\I{i=k-1}-\I{i=\ell}+\I{i=\ell+1}.\]
\end{compactenum}
Suppose $n'_r, n_r\ge 1$, and let $T,T'$ be uniformly random picks respectively from $\cT(\nseq)$ and from $\cT(\nseq')$ conditional on $\rk_\varnothing(T)=\rk_\varnothing(T')=r$. 
If $\nseq\preceq\nseq'$ then $\H(T)$ is stochastically dominated by $\H(T')$.
\end{thm}
In words, to obtain $\nseq$ from $\nseq'$ in the definition of $\preceq$, we either exclude a vertex with degree $1$, or, for $1\le k\le \ell $ we replace one vertex with $k$ children and one vertex with $\ell$  children by a vertex with $k-1$ children and one with $\ell+1$ children. In the second case, informally, the degrees in $\nseq$ are more skewed than the degrees in $\nseq'$. We prove Theorem~\ref{thm:stoch_ord} in Appendix~\ref{app:stochdom}. The proof is a straightforward adaptation of the proof of Theorem~9 in \cite{DonLAB24} from labelled rooted trees to plane trees with fixed root degree.

\begin{proof}[Proof of Proposition \ref{prop:stoch_dom} using Theorem~\ref{thm:stoch_ord}]
We will show that for any type sequence $\nseq=( n_i)_{i\ge 0}$ with size at most $n+1$ and $ n_r\ge 1$, it holds that $\nseq\preceq \nseq^*$. We may assume without loss of generality that the size of $\nseq$ equals $n+1$. Indeed, if its size were $n+1-k$ for $k\ge 1$ we may instead consider $\nseq'=(n'_i)_{i\ge 0}$ with $n'_i= n_i+k\I{ i=1}$. Then $\nseq'$ has size $n+1$, $n'_r\ge 1$ and $\nseq \preceq \nseq'$.

Note that for any $i\geq 2$, we have $ n_i^*=\I{i=r}\leq n_i$ so $\sum_{i\ge 2}(i-1)( n_i-n^*_i)\ge 0$. We will prove by induction over $m\ge 0$ that for all $m\ge 0$, if $\nseq=( n_i)_{i\ge 0}$ is a type sequence with size $n+1$, $ n_r\ge 1$ and $\sum_{i\ge 2}(i-1)(n_i-n^*_i)=m$, then it holds that $\nseq \preceq \nseq^*$.  

 First, we see that if $\sum_{i\ge 2}(i-1)( n_i-n^*_i)=0$ then $ n_i= n_i^*$ for all $i\geq 2$, which then implies that $ n_0+ n_1= n_0^*+ n_1^*$ and $ n_1= n_1^*$ because $\nseq$ and $\nseq^*$ have the same size. Thus, $\nseq=\nseq^*$ and the statement holds for $m=0$.

Let $m\ge 1$ and suppose we have shown the statement for all $0\le k<m$. Let $\nseq=( n_i)_{i\ge 0}$ be a type sequence with size $n+1$, $ n_r\ge 1$ and $\sum_{i\ge 2}(i-1)( n_i-n^*_i)=m$. We will define a type sequence $\nseq'=(n'_i)_{i\ge 0}$ of size $n+1$ with $ n_r'\geq 1$ for which $\nseq'$ covers $\nseq$ and $\sum_{i\ge 2}(i-1)(n'_i-n^*_i)=m-1$. Then $\nseq\preceq \nseq^*$ follows by the induction hypothesis. Since $\sum_{i\ge 2}(i-1)( n_i-n^*_i)\ge 1$, there is an $\ell\ge 2$ with $ n_\ell>n^*_{\ell}$. Moreover, since $\nseq$ is a type sequence, we have $ n_0>0$. Then, set 
\[ n'_i =  n_i-\I{i=0}+\I{i=1}-\I{i=\ell}+\I{i=\ell-1}\]
so that we obtain $\nseq'$ from $\nseq$ by replacing a vertex with degree $0$ and a vertex with degree $\ell$ by a vertex with degree $1$ and a vertex with degree $\ell-1$. 
It is clear that $\nseq'$ satisfies the claimed properties, which finishes the proof.
\end{proof}

\section{Symmetrization by shuffling}
\label{sec:symmetrization}

\subsection{Shuffling the order of plane trees}

Via the lexicographic order, the (plane) trees we consider are naturally endowed with a left-right order which is unrelated to their genealogical structures. We formalise this principle by considering a family of bijective maps that rearrange the order of the children of each vertex of an input tree. Recall that for $k\geq 1$, $\cS_k$ stands for the set of permutations of $[k]$.

\begin{definition}
\label{shuffle_def}
A \emph{shuffle} $\Gamma=(\gamma_{u,k}\, ;\, u\in\mathbb{U},k\geq 1)$ is a collection of permutations with $\gamma_{u,k}\in \cS_k$ for all $u\in \mathbb{U},k\ge 1$.
\end{definition}

Let $t$ be a tree. We inductively construct $\Gamma_u(t)\in\mathbb{U}$ for all $u\in t$ by setting
\begin{equation}
\label{shuffle_formula}
\Gamma_\varnothing(t)=\varnothing\quad\text{ and }\quad\Gamma_{v*(j)}(t)=\Gamma_v(t)*(\gamma_{v,\rk_v(t)}(j))
\end{equation}
for any $v\in t$ and $1\leq j\leq \rk_v(t)$. Then, we define $\Gamma(t)=\{\Gamma_u(t)\, :\, u\in t\}$. Thus, informally, if we place a mark on vertex $u$ in $t$, then rearrange the order of the children of each vertex $v$ in $t$ according to $\gamma_{v,\rk_v(t)}$, then $\Gamma_u(t)$ is the vertex with a mark in the resulting tree $\Gamma(t)$. See Figure~\ref{fig:shuffle}.
\begin{figure}[h]
    \centering
    \begin{subfigure}{0.4\textwidth}
    \centering
    \begin{tikzpicture}[baseline={(current bounding box.center)}, every node/.style={circle,fill,inner sep=0pt, minimum size=5pt}, level
distance=10mm,level 1/.style={sibling distance=20mm},
     level 2/.style={sibling distance=10mm},
     level 3/.style={sibling distance=5mm}]
     \node  {}  [grow=up]
     child {node{}
     	child {node{}
                child {node{}
                    child {node{}}
                }
     		child {node{}}
     		child {node{}
     			child {node{}}
     			child {node{}}
     			}
     		}
     	child {node{}
     		child {node{}}
                }
     	}
     child {node{}};
    \end{tikzpicture}
    \end{subfigure}
    \begin{subfigure}{0.4\textwidth}
    \centering
    \begin{tikzpicture}[baseline={(current bounding box.center)}, every node/.style={circle,fill,inner sep=0pt, minimum size=5pt}, level
distance=10mm,level 1/.style={sibling distance=20mm},
     level 2/.style={sibling distance=10mm},
     level 3/.style={sibling distance=5mm}]
     \node  {} [grow=up]
     child {node{}
            child {node{}
     		child {node{}}
                }
     	child {node{}
                child {node{}
     			child {node{}}
     			child {node{}}
     			}
                child {node{}
                    child {node{}}
                }
     		child {node{}}
     		}
     	}
     child {node{}};
    \end{tikzpicture}
    \end{subfigure}
    \caption{Left: an example of a tree $t$. Right: an example for $\Gamma(t)$, where $\Gamma$ is some shuffle.}
    \label{fig:shuffle}
\end{figure}
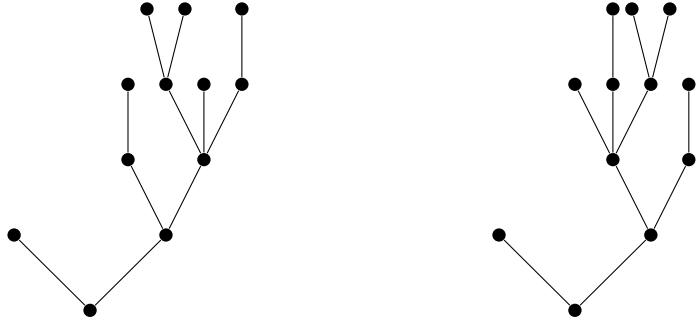

The next result gathers some properties of the map $u\mapsto \Gamma_u(t)$ for a fixed tree $t$.

\begin{prop}
\label{properties_shuflle}
For any shuffle $\Gamma$ and any tree $t$, the following assertions hold.
\begin{compactenum}
    \item[$(i)$] For all $u\in t\setminus\{\varnothing\}$, the parent of $\Gamma_u(t)$ is $\Gamma_{\overleftarrow{u}}(t)$.
    \item[$(ii)$] For all $u\in t$, it holds that $|\Gamma_u(t)|=|u|$.
    \item[$(iii)$] The map $u\in t\longmapsto \Gamma_u(t)\in\Gamma(t)$ is a bijection.
    \item[$(iv)$] The set $\Gamma(t)$ is a tree.
    \item[$(v)$] For all $u\in t$, it holds that $\rk_{\Gamma_u(t)}(\Gamma(t))=\rk_u(t)$.
    \item[$(vi)$] For all $u,v\in t$, it holds that $\Gamma_u(t) \wedge \Gamma_v(t)=\Gamma_{u\wedge v}(t)$.
\end{compactenum}
\end{prop}

\begin{proof}
The point $(i)$ readily follows from the definition (\ref{shuffle_formula}). Then, relying on $(i)$, an obvious induction on the height of $u$ yields $(ii)$.

By $(ii)$, the proof of $(iii)$ only requires to show that $u\in\{v\in t\, :\, |v|=h\}\longmapsto \Gamma_u(t)$ is injective for all $h\geq 0$. We do so by induction. The case $h=0$ is obvious because $\varnothing$ is the only vertex of $t$ with height zero. For $h\geq 1$, let $u,v\in t$ such that $|u|=|v|=h$ and $\Gamma_u(t)=\Gamma_v(t)$. By $(i)$ we have $\Gamma_{\overleftarrow{u}}(t)=\Gamma_{\overleftarrow{v}}(t)$ because both are the parent of the same vertex, so $\overleftarrow{u}=\overleftarrow{v}$ by induction hypothesis. Writing $u=w*(i)$ and $v=w*(j)$, we get that $\gamma_{w,\rk_w(t)}(i)=\gamma_{w,\rk_w(t)}(j)$ and then $i=j$ since $\gamma_{w,\rk_w(t)}$ is a permutation. Thus, we indeed have $u=v$ as desired and $(iii)$ follows.

Recall that $\Gamma_{\varnothing}(t)=\varnothing$ by definition, so that $\Gamma(t)$ satisfies $(a)$ in Definition~\ref{def:tree} of trees. $(b)$ follows from $(i)$. Moreover, by $(i)$ and $(iii)$, for any $u\in t$, the children of $\Gamma_u(t)$ in $\Gamma(t)$ are the $\Gamma_{u*(j)}(t)$ for $1\leq j\leq \rk_u(t)$.By $(iii)$, these vertices are distinct. Then, the fact that $\gamma_{u,\rk_u(t)}$ is a permutation of $[\rk_u(t)]$ ensures that the children of $\Gamma_u(t)$ in $\Gamma(t)$ are exactly the $\Gamma_u(t)*(j)$ for $1\leq j\leq \rk_u(t)$. This completes the proof of $(iv)$, because $\Gamma(t)$ satisfies Definition~\ref{def:tree} $(c)$, and also gives $(v)$.

Still using $(i)$, an easy induction entails that for all $u,v\in t$, if $u\preceq v$ then $\Gamma_u(t)\preceq \Gamma_v(t)$. The converse holds by injectivity thanks to $(iii)$. Now, let $u,v\in t$. The point $(iv)$ ensures that there is $w\in t$ such that $\Gamma_u(t)\wedge\Gamma_v(t)=\Gamma_w(t)$. By definition of the most recent common ancestor, the previous equivalence and the point $(ii)$ yield that $w$ is the common ancestor of $u,v$ with maximum height, i.e.~$w=u\wedge v$.
\end{proof}

Now, we consider the properties of $\Gamma$ as a transformation of trees.

\begin{prop}
\label{shuffle_bij_type}
For any shuffle $\Gamma$, the map $t\in\cT\longmapsto\Gamma(t)\in\cT$ is a bijection. Moreover, for any tree $t$, $\Gamma(t)$ has the same type as $t$.
\end{prop}

\begin{proof}
The fact that $\Gamma(t)$ and $t$ have the same type follows from Proposition~\ref{properties_shuflle} $(iii)$ and $(v)$. Hence, we only need to prove that for all $n\geq 1$, for any shuffle $\Gamma$, the restriction of $t\in\cT\longmapsto\Gamma(t)\in\cT$ to the subset of all trees with size $n$ is injective. Indeed, as a self-map of a finite set, it is bijective if and only if it is injective. We prove this by induction. The singleton $\{\varnothing\}$ is the only tree of size $1$ so the case $n=1$ is obvious.

Let $n\geq 2$, let $\Gamma$ be a shuffle, and let $t,t'$ be two trees of size $n$ such that $\Gamma(t)=\Gamma(t')$. Thanks to Proposition~\ref{properties_shuflle} $(v)$, we readily get that $\rk_\varnothing(t)=\rk_\varnothing(t')$. Next, for any $j\geq 1$, we set $\Gamma^{(j)}=(\gamma_{(j)*u,k}\, ;\, u\in\mathbb{U},k\geq 1)$ and note that this is a shuffle. Then, we check from (\ref{shuffle_formula}) that if $i=\gamma_{\varnothing,\rk_\varnothing(t)}(j)$, with $1\leq i,j\leq \rk_\varnothing(t)$, then $\theta_{(i)}\, \Gamma(t)=\Gamma^{(j)}(\theta_{(j)}t)$. Doing the same with $t'$ gives that $\Gamma^{(j)}(\theta_{(j)}t)=\Gamma^{(j)}(\theta_{(j)}t')$ for all $1\leq j\leq \rk_\varnothing(t)$. By induction hypothesis, it follows that $\theta_{(j)}t=\theta_{(j)}t'$ for all $1\leq j\leq \rk_\varnothing(t)=\rk_\varnothing(t')$ and thus that $t=t'$.
\end{proof}

An important example of shuffle is the one that reverses the order of the children of every vertex. More precisely, we consider the \emph{mirroring shuffle} $\rM=(\mathrm{m}_{u,k}\, ;\, u\in\mathbb{U},k\geq 1)$, where
\begin{equation}
\label{mirror}
\forall j\in[k],\quad \mathrm{m}_{u,k}(j)=k+1-j,
\end{equation}
for any $u\in\mathbb{U}$ and $k\geq 1$. If $t$ is a tree and $u\in t$, then we call $\rM(t)$ the \emph{mirror image of $t$} and we call $\rM_u(t)$ the \emph{mirror image of $u$ in $t$}. See Figure~\ref{fig:mirror}.

\begin{figure}[h]
    \centering
    \begin{subfigure}{0.4\textwidth}
    \centering
    \begin{tikzpicture}[baseline={(current bounding box.center)}, every node/.style={circle,fill,inner sep=0pt, minimum size=5pt}, level
distance=10mm,level 1/.style={sibling distance=20mm},
     level 2/.style={sibling distance=10mm},
     level 3/.style={sibling distance=5mm}]
     \node  {} [grow=up]
     child {node{}
     	child {node{}
                child {node{}
                    child {node{}}
                }
     		child {node{}}
     		child {node{}
     			child {node{}}
     			child {node{}}
     			}
     		}
     	child {node{}
     		child {node{}}
                }
     	}
     child {node{}};
    \end{tikzpicture}
    \end{subfigure}
    \begin{subfigure}{0.4\textwidth}
    \centering
    \begin{tikzpicture}[baseline={(current bounding box.center)}, every node/.style={circle,fill,inner sep=0pt, minimum size=5pt}, level
distance=10mm,level 1/.style={sibling distance=20mm},
     level 2/.style={sibling distance=10mm},
     level 3/.style={sibling distance=5mm}]
     \node  {} [grow=up]
     child {node{}}
     child {node{}
            child {node{}
     		child {node{}}
                }
     	child {node{}
                child {node{}
     			child {node{}}
     			child {node{}}
     			}
                child {node{}}
                child {node{}
                    child {node{}}
                    }
     		}
     	}
      ;
    \end{tikzpicture}
    \end{subfigure}
    \caption{Left: an example of a tree $t$. Right: the tree $\rM(t)$.}
    \label{fig:mirror}
\end{figure}
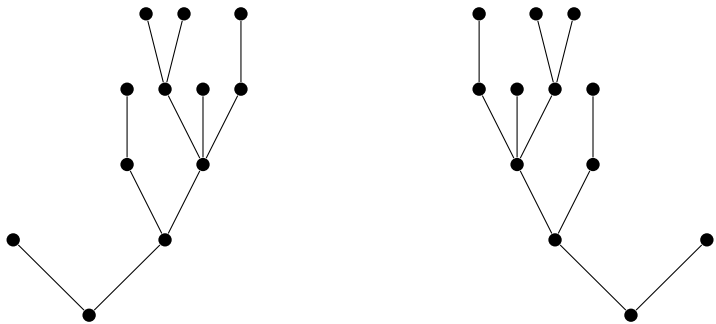

The next proposition asserts that the mirroring shuffle inverts the lexicographic order for vertices that are not ancestrally related.

\begin{prop}
\label{mirror_reverses}
Let $t$ be a tree and let $u<v\in t$. If $u \prec v$ then $\rM_u(t)< \rM_v(t)$. Otherwise, $\rM_u(t)>\rM_v(t)$. 
\end{prop}

\begin{proof}
First, observe that $\rM_u(t)\ne \rM_v(t)$ follows from Proposition~\ref{properties_shuflle} $(iii)$. 

We start with the first statement. By Proposition~\ref{properties_shuflle} $(vi)$ it holds that $\rM_u(t)\wedge \rM_v(t)=\rM_{u\wedge v}(t)=\rM_u(t)$ so that $\rM_u(t)\prec \rM_v(t)$ and thus $\rM_u(t)<\rM_v(t)$. 

For the second statement, let us write $w=u\wedge v$ and $k=\rk_w(t)$. By assumption, there are distinct $1\leq i,j\leq k$ such that $w*(i)\preceq u$ and $w*(j)\preceq v$. Note that $u<v$ implies that $i<j$. Thanks to (\ref{shuffle_formula}) and Proposition~\ref{properties_shuflle} $(vi)$, we have $\rM_w(t)*(k+1-i)\preceq \rM_u(t)$ and $\rM_w(t)*(k+1-j)\preceq \rM_v(t)$. Thus, we have $\rM_v(t)<\rM_u(t)$.
\end{proof}

We end this section by deducing from Proposition~\ref{mirror} the effect of the mirroring shuffle on the spinal weights $\rS_u^{\rg}(t),\rS_u^{\rd}(t),\rS_u(t)$ introduced in Section~\ref{sec:construction_around_spine}. Proposition~\ref{mirror_reverses} and Proposition~\ref{properties_shuflle} $(i)$, $(iii)$, $(v)$ and $(vi)$ lead that for all $u\in t$, we have
\begin{equation}
\label{spine-weight_mirror}
\rS_{\rM_u(t)}(\rM(t))=\rS_u(t), \quad \rS_{\rM_u(t)}^{\rd}(\rM(t))=\rS_u^{\rg}(t)\quad\text{ and }\quad \rS_{\rM_u(t)}^{\rg}(\rM(t))=\rS_u^{\rd}(t).
\end{equation}

\subsection{Proof of Propositions~\ref{left-right_spine_weights} and \ref{left-right_spine_height}}

Let $\nseq=(n_i)_{i\geq 0}$ be a type sequence of size $n\geq 3$ and let $T$ be uniformly distributed on $\cT(\nseq)$. Armed with the properties of shuffles obtained above, we now show Propositions~\ref{left-right_spine_weights} and \ref{left-right_spine_height}, which were used in the proof of Theorem~\ref{thm:nseq} to remove the asymmetry introduced by Proposition~\ref{bound_product_right-weight-height}.

Proposition~\ref{left-right_spine_weights} asserts that in a uniformly random tree with a given type, if there is a vertex with a large spinal weight, then it is likely that there is also a vertex of which both the left and right spinal weight are quite large. Informally, this is because the vertices that attach to the spine of node $u$ do so independently and indifferently to the left or right. We utilise this idea by shuffling $T$ in a uniform manner.

\begin{proof}[Proof of Proposition~\ref{left-right_spine_weights}]
Let $\Pi=(\pi_{w,k}\, ;\, w\in\mathbb{U},k\geq 1)$ be a family of independent random permutations such that $\pi_{w,k}$ is a uniformly random permutation of $[k]$ for any $w\in\mathbb{U},k\geq 1$. Thus, $\Pi$ is a random shuffle in the sense of Definition~\ref{shuffle_def}. We also assume that $\Pi$ and $T$ are independent. As Proposition~\ref{shuffle_bij_type} asserts that any shuffle induces a type-preserving bijection of the set of trees, we see that $\Pi(T)$ is uniformly distributed on $\cT(\nseq)$ and so has the same law as $T$. Thus, the independence between $T$ and $\Pi$ entails that
\begin{multline*}
\P\Big(\max_{u\in T}\varepsilon\rS_u(T)\geq 1 \, ;\, \max_{u\in T}\varepsilon\rS_u(T)\geq \max_{u\in T}\min\big(\rS_u^{\rd}(T),\rS_u^{\rg}(T)\big)\Big)\\
=\sum_{t\in\cT}\P(T=t)\P\Big(\max_{u\in \Pi(t)}\varepsilon\rS_u(\Pi(t))\geq 1 \vee \max_{u\in \Pi(t)}\min\big(\rS_u^{\rd}(\Pi(t)),\rS_u^{\rg}(\Pi(t))\big)\Big).
\end{multline*}
Let $t$ be a tree and let $w\in t$ such that $\rS_w(t)=\max_{u\in t}\rS_u(t)$. Thanks to the expression of $\rS_u(T)$ given by (\ref{spinal-weights_as_sums}), Proposition~\ref{properties_shuflle} $(i)$, $(iii)$, and $(vi)$ yield that $\rS_{\Pi_u(t)}(\Pi(t))=\rS_u(t)$ for any $u\in t$, and then that $\max_{u\in \Pi(t)}\rS_u(\Pi(t))=\max_{u\in T}\rS_u(T)=\rS_w(t)=\rS_{\Pi_w(t)}(\Pi(t))$. In particular, if it holds that $\max_{u\in \Pi(t)}\varepsilon\rS_u(\Pi(t))\geq \max_{u\in \Pi(t)}\min\big(\rS_u^{\rd}(\Pi(t)),\rS_u^{\rg}(\Pi(t))\big)$, then we have $\varepsilon\rS_w(t)\geq \min\big(\rS_{\Pi_w(t)}^{\rd}(\Pi(t)),\rS_{\Pi_w(t)}^{\rg}(\Pi(t))\big)$. Therefore, we obtain the inequality
\begin{multline}
\label{left-right_spine_step}
\P\Big(\max_{u\in T}\varepsilon\rS_u(T)\geq 1 \, ;\, \max_{u\in T}\varepsilon\rS_u(T)\geq \max_{u\in T}\min\big(\rS_u^{\rd}(T),\rS_u^{\rg}(T)\big)\Big)\\
\leq \sup_{\substack{t\in\cT,w\in t\\ \varepsilon\rS_w(t)\geq 1}}\left[\P\big(\rS_{\Pi_w(t)}^{\rd}(\Pi(t))\leq \varepsilon\rS_w(t)\big)+\P\big(\rS_{\Pi_w(t)}^{\rg}(\Pi(t))\leq \varepsilon\rS_w(t)\big)\right]
\end{multline}
with the convention that $\sup\emptyset=0$.

Now, let us fix a tree $t$ and a vertex $w\in t$ such that $\varepsilon\rS_w(t)\geq 1$. We write $|w|=m$, $w=(j_1,\ldots,j_m)$. Then for any $0\leq i\leq m-1$, we denote by $w_i=(j_1,\ldots,j_i)$ the ancestor of $w$ at height $i$ and we set $a_i=\rk_{w_i}(t)-1\geq 0$ and $J_i=\pi_{w_i,\rk_{w_i}(t)}(j_{i+1})-1\geq 0$. With this notation, (\ref{spinal-weights_as_sums}) can now be written as
\begin{equation}
\label{uniform-shuffle_spinal-weights}
\rS_w(t)=\sum_{i=0}^{m-1}a_i,\quad\rS_{\Pi_w(t)}^{\rg}(\Pi(t))=\sum_{i=0}^{m-1} J_i\quad\text{ and }\quad \rS_{\Pi_w(t)}^{\rd}(\Pi(t))=\sum_{i=0}^{m-1} (a_i-J_i).
\end{equation}
We know from the law of $\Pi$ that the $J_i,0\leq i\leq m-1$ are independent and respectively uniform on $\{0,\ldots, a_i\}$. In particular, we compute that $\E{J_i}=a_i/2$ and $\V{J_i}=(a_i^2+2a_i)/12\leq a_i^2/4$. Setting $a_{\max}=\max_{0\leq i\leq m-1} a_i$, we pursue the proof with two different methods depending on the order of $a_{\max}$. We also assume that $\varepsilon<1/4$ since we otherwise have that $2\sqrt{8\varepsilon}\geq 1$, which would readily give the desired result.

\smallskip
$\bullet$ If $a_{\max}\geq \rS_w(t)\sqrt{\varepsilon/2}$, then we write that $\P\big(\rS_{\Pi_w(t)}^{\rg}(\Pi(t))\leq \varepsilon \rS_w(t)\big)\leq \P(J_i\leq\varepsilon\rS_w(t))$ for any $0\leq i\leq m-1$. Since the law of $J_i$ is uniform on $\{0,\ldots,a_i\}$, choosing $i$ such that $a_i=a_{\max}$ gives $\P(J_i\leq \varepsilon\rS_w(t))\leq (1+\varepsilon\rS_w(t))/(1+a_{\max})$. We recall that $1\leq \varepsilon\rS_w(t)$ to find
\[\P\big(\rS_{\Pi_w(t)}^{\rg}(\Pi(t))\leq \varepsilon \rS_w(t)\big)\leq \frac{1+\varepsilon\rS_w(t)}{1+a_{\max}}\leq \frac{2\varepsilon\rS_w(t)}{a_{\max}}\leq \sqrt{8\varepsilon}.\]

\smallskip
$\bullet$ If $a_{\max}\leq \rS_w(t)\sqrt{2\varepsilon}$, then we want to apply Chebyshev's inequality. To do this, we observe that $\rS_w(t)=\sum_{i=0}^{m-1}a_i=2\sum_{i=0}^{m-1}\E{J_i}$. Since $\varepsilon<1/4$, it follows that
\[\P\big(\rS_{\Pi_w(t)}^{\rg}(\Pi(t))\leq \varepsilon\rS_w(t)\big)\leq \P\Big(\Big|\sum_{i=0}^{m-1} J_i-\E{J_i}\Big|\geq \tfrac{1}{4}\rS_w(t)\Big)\leq \frac{16}{\rS_w(t)^2}\V{\sum_{i=0}^{m-1} J_i}.\]
Then, the independence of the $J_i$ yields that
\[\P\big(\rS_{\Pi_w(t)}^{\rg}(\Pi(t))\leq \varepsilon\rS_w(t)\big)\leq \frac{4}{\rS_w(t)^2}\sum_{i=0}^{m-1}a_i^2\leq \frac{4a_{\max}}{\rS_w(t)}\leq \sqrt{8\varepsilon}.\]

\smallskip
Finally, we have shown that for any tree $t$ and any $w\in t$ such that $\varepsilon\rS_w(t)\geq 1$ it holds that $\P\big(\rS_{\Pi_w(t)}^{\rg}(\Pi(t))\leq \varepsilon\rS_w(t)\big)\leq \sqrt{8\varepsilon}$. It is easy to see from (\ref{uniform-shuffle_spinal-weights}) that $\rS_{\Pi_w(t)}^{\rd}(\Pi(t))$ has the same law as $\rS_{\Pi_w(t)}^{\rg}(\Pi(t))$, so we also have $\P\big(\rS_{\Pi_w(t)}^{\rd}(\Pi(t))\leq \varepsilon\rS_w(t)\big)\leq \sqrt{8\varepsilon}$. Combining these bounds with (\ref{left-right_spine_step}) entails the desired inequality.
\end{proof}

Proposition~\ref{left-right_spine_height} basically expresses that in a uniformly random tree with fixed type, the pair of vertices that achieves the second-order height has the same law in the left-to-right depth-first ordering as in the right-to-left depth-first ordering of the tree. More precisely, our proof make use of the fact that the mirroring shuffle preserves the law of $T$. 

\begin{proof}[Proof of Proposition~\ref{left-right_spine_height}]
Let $s>0$. The formula (\ref{second-height_alt2}) for the second-order height, given in Proposition~\ref{second-height_alt}, yields that
\[\P\Big(\H^{(2)}(T)\max_{u\in T}\min\big(\rS_u^{\rd}(T),\rS_u^{\rg}(T)\big)\geq s\Big)\leq \P\Big(\max_{u,v\in T} (|v|-|u\wedge v|)\min\big(\rS_u^{\rd}(T),\rS_u^{\rg}(T)\big)\geq s\Big).\]
We now want to make the restriction $u\leq v$ appear by showing that
\begin{equation}
\label{left-right_spine_tool}
\xi:=\P\Big(\max_{\substack{u,v\in T\\ v<u}} (|v|-|u\wedge v|)\min\big(\rS_u^{\rd}(T),\rS_u^{\rg}(T)\big)\geq s\Big)\leq \P\Big(\max_{\substack{u,v\in T\\ u\leq v}} (|v|-|u\wedge v|)\rS_u^{\rd}(T)\geq s\Big).
\end{equation}
For $u,v\in t$, if $|v|-|u\wedge v|>0$ then $v$ cannot be an ancestor of $u$. As $s>0$, we thus have
\[\xi=\P\Big(\exists u,v\in T\, :\, v<u,u\wedge v\neq v,(|v|-|u\wedge v|)\min\big(\rS_u^{\rd}(T),\rS_u^{\rg}(T)\big)\geq s\Big).\]
We know from Proposition~\ref{properties_shuflle} $(ii)$ and $(vi)$ that maps that are induced by shuffles on trees, including the mirroring shuffle $\rM$ from (\ref{mirror}), preserve the height and the genealogical relations. Thus, for all $u,v\in t$, it holds that $|\rM_v(t)|-|\rM_u(t)\wedge \rM_v(t)|=|v|-|u\wedge v|$. Moreover, these maps are injective by Proposition~\ref{properties_shuflle} $(iii)$, so $u\wedge v\neq v$ if and only if $\rM_u(t)\wedge\rM_v(t)\neq\rM_v(t)$. More specifically to the mirroring shuffle, (\ref{spine-weight_mirror}) yields that $\min\big(\rS_{\rM_u(T)}^{\rg}(\rM(T)),\rS_{\rM_u(T)}^{\rd}(\rM(T))\big)=\min\big(\rS_u^{\rd}(T),\rS_u^{\rg}(T)\big)$. Furthermore, Proposition~\ref{mirror_reverses} ensures that $u\wedge v\neq v$ and $v<u$ if and only if $\rM_u(t)\wedge\rM_v(t)\neq\rM_{u}(t)$ and $\rM_u(t)<\rM_v(t)$. All these arguments, together with the fact that $u\in t\mapsto \rM_u(t)\in\rM(t)$ is a surjection, as given by Proposition~\ref{properties_shuflle} $(iii)$, allow us to obtain
\[\xi=\P\Big(\exists u,v\in\rM(T): u<v, u\wedge v\neq u, (|v|-|u\wedge v|)\min\big(\rS_u^{\rd}(\rM(T)),\rS_u^{\rg}(\rM(T))\big)\geq s\Big).\]
Next, we recall from Proposition~\ref{shuffle_bij_type} that the shuffle $\rM$ induces a type-preserving transformation on the set of all trees. This yields that $\rM(T)$ is uniformly distributed on $\cT(\nseq)$, as $T$ is, and so that
\[\xi=\P\Big(\exists u,v\in T\, :\,u<v, u\wedge v\neq u,(|v|-|u\wedge v|)\min\big(\rS_u^{\rd}(T),\rS_u^{\rg}(T)\big)\geq s\Big).\]
This readily entails the desired (\ref{left-right_spine_tool}), and hence concludes the proof.
\end{proof}

\section*{Acknowledgements}

This work was initiated during the 2024 \emph{Probability and Combinatorics Workshop} at McGill’s Bellairs Research Institute, in Holetown, Barbados. We thank the
organisers for inviting us and for creating a stimulating research environment.
SD acknowledges the financial support of the CogniGron research center and the Ubbo Emmius Funds (Univ. of Groningen).

\appendix

\section{Stochastic domination of the height of random plane trees}\label{app:stochdom}

In this section, we prove Theorem~\ref{thm:stoch_ord}. The proof is a straightforward adaptation of the proof of Theorem~9 in \cite{DonLAB24} from labelled rooted trees to plane trees with fixed root degree. 

 We first need some notation. Recall that for any integer $n\geq 1$, $[n]$ stands for $\{1,\ldots,n$\}. Let $t$ be a plane tree with at least two vertices and $u,v\in t$ with $u< v$. (We advice the reader to look at the examples in Figure~\ref{fig:cuttree} before reading the formal definitions below.)

 \begin{figure}
 \centering
     \begin{subfigure}{0.8\linewidth}
      \centering
         \includegraphics[scale=0.4, page=1]{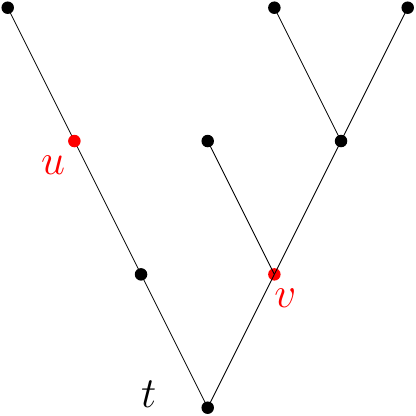}
         \hfill
        \includegraphics[scale=0.4, page=2]{eq_class1.pdf}
        \hfill
         \includegraphics[scale=0.4, page=3]{eq_class1.pdf}
         \caption{A tree $t$ with vertices $u<v$ with $u\not\prec v$ marked red, depicted with $t_{u,v}$ and $f_{u,v}(t)$.  \label{fig:cuttree1}}
     \end{subfigure}\\
     \begin{subfigure}{0.8\linewidth}
      \centering
         \includegraphics[scale=0.4, page=1]{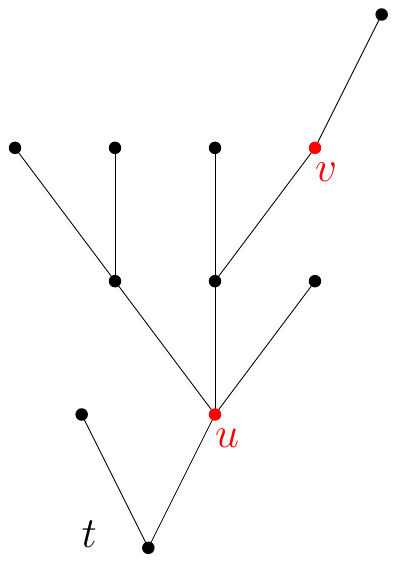}
         \hfill
        \includegraphics[scale=0.4, page=2]{eq_class2.pdf}
        \hfill
         \includegraphics[scale=0.4, page=3]{eq_class2.pdf}
         \caption{A tree $t$ with vertices $u<v$ with $u\prec v$ marked red, depicted with $t_{u,v}$ and $f_{u,v}(t)$. Observe that removing  pendant trees from $u$ has changed the label of the right-most red vertex from $v=222$ in $t$ to $212$ in $t_{u,v}$. In fact, all labels of vertices rooted at $u$ in $t_{u,v}$ have changed.\label{fig:cuttree2} }
     \end{subfigure}
     \caption{A depiction of $g(t,\{u,v\})$ for two trees $t$ and marked vertices $u,v\in t$ with $u<v$. Then $g(t,\{u,v\})$ is the array that contains the tree $t_{u,v}$ obtained by removing the unmarked subtrees rooted at children of $u$ and $v$ from $t$, the labels of the marked vertices in $t_{u,v}$ and the multiset $f_{u,v}(t)$ of the unmarked subtrees rooted at the children of $u$ and $v$. \label{fig:cuttree}}

 \end{figure}
\begin{itemize}
    \item If $u\not\prec v$ (i.e.~there is no $w\in \U \setminus\{\varnothing\}$ such that $v=u* w$) let $f_{u,v}(t)$ be the multiset of $\rk_u(t)+\rk_v(t)$ plane trees rooted at the children of $u$ and $v$ and let $t_{u,v}$ be the tree obtained by removing the trees rooted at the children of $u$ and $v$. (Then $f_{u,v}(t)\cup \{ t_{u,v}\}$ consists of the plane trees obtained by removing the edges between $u$ and $v$ and their children.) See Figure~\ref{fig:cuttree1}. Formally, for any $t'\in \cT$, let the multiplicity of $t'$ in $f_{u,v}(t)$ equal
    \[|\{i\in [\rk_u(t)]:\theta_{u*(i)}t=t'\}|+|\{i\in [\rk_v(t)]:\theta_{v*(i)}t=t'\}|\]
    and let 
    \[t_{u,v}=\{x\in t:v\not \prec x\text{ and }u\not \prec x\}.\]
    Let 
    \[g(t,\{u<v\})=\left(t_{u,v},\{u,v\},f_{u,v}(t)\right).\]
     \item If $u\prec v$ (i.e.~there is a $w\in \U\setminus \{\varnothing\}$ such that $v=u* w$) let $f_{u,v}(t)$ be the multiset of the $\rk_v(t)$ plane trees rooted at the children of $v$ and the $\rk_u(t)-1$ plane trees rooted at the children of $u$ that do not contain $v$, and let $t_{u,v}$ be the tree obtained by removing the trees rooted at the children of $v$ and the trees rooted at the children of $u$ that do not contain $v$.  (Then $f_{u,v}(t)\cup \{ t_{u,v} \}$ consists of the plane trees obtained by removing  the edges between $v$ and its children and the edges between $u$ and its children, except the child that is an ancestor of $v$.) See Figure~\ref{fig:cuttree2}. Formally, if $v=u*(w_1,\dots,w_m)$, for any $t'\in \cT$, let the multiplicity of $t'$ in $f_{u,v}(t)$ equal
    \[|\{i\in [\rk_u(t)]\setminus\{w_1\}:\theta_{u*(i)}t=t'\}|+|\{i\in [\rk_v(t)]:\theta_{v*(i)}t=t'\}|\]
    and let $t_{u,v}$ be the tree characterised by the following.
    \begin{enumerate}
    \item[$(a)$] For $x\in \U$ such that $u\not \prec x$, $x\in t_{u,v}$ if and only if $x\in t$. 
    \item[$(b)$] For $x\in \U$ such that $u \prec x$, $x\in t_{u,v}$ if and only if $x=u*(1)*y$ for some $y\in\U$ with $u*(w_1)*y\in t$ and $v\not\prec u*(w_1)*y$. 
    \end{enumerate}
    Let 
    \[g(t,\{u<v\})=\left(t_{u,v},\{u,u*(1,w_2,\dots, w_m)\},f_{u,v}(t)\right).\]
\end{itemize}  
\smallskip

For trees $t, t'$, distinct $u,v\in t$ and distinct $u',v'\in t'$, we say that $(t,\{u,v\})\sim (t',\{u',v'\})$ when $g(t,\{u,v\})=g(t',\{u',v'\})$. (Note that if, without loss of generality, $u<v$ and $u'<v'$ then $g(t,\{u,v\})=g(t',\{u',v'\})$  means that $u=u'$ and that, informally, $t'$ can be obtained from $t$ by detaching the trees rooted at $u$ that do not contain $v$ and the trees rooted at $v$ and then reattaching a subset of these trees at $u$ in some order and the trees not in this subset at $v$ in some order; if we imagine placing a mark at vertex $v$ before this procedure, then after detaching and reattaching subtrees this mark is now at vertex $v'$.)  Observe that if $(t,\{u,v\})\sim (t',\{u',v'\})$ then the type sequences of $t$ and $t'$ agree up to potentially replacing a vertex of degree $a$ and a vertex of degree $b$ by a vertex of degree $c$ and a vertex of degree $d$ for some $a,b,c,d$ that satisfy that $a+b=c+d$, and, in that case, $\{\rk_u(t),\rk_v(t)\}=\{a,b\}$ and $\{\rk_{u'}(t'),\rk_{v'}(t')\}=\{c,d\}$.

\medskip
Let $\nseq=( n_i)_{i\ge 0}$ and $\nseq'=(n'_i)_{i\ge 0}$ be type sequences of size $n\ge 3$ with $ n_r,n'_r\ge 1$. Suppose that $\nseq'$ covers $\nseq$ so that there exist $k, \ell $ with $1\le k\le \ell \le n-1 $ such that 
    \[  n_i = n'_i -\I{i=k} +\I{i=k-1}-\I{i=\ell}+\I{i=\ell+1}.\]
 Let $T$ be a uniformly random element in $\cT(\nseq)$ conditional on $\rk_\varnothing(T)=r$ and let $T'$ be a uniformly random element in $\cT(\nseq')$ conditional on $\rk_\varnothing(T')=r$. Let $U$ be a uniformly random vertex unequal to $\varnothing$ of degree $k-1$ in $T$, let $V$ be a uniformly random vertex unequal to $\varnothing$ and $U$ of degree $\ell+1$ in $T$, let $U'$ be a uniformly random vertex unequal to $\varnothing$ of degree $k$ in $T'$ and let $V'$ uniformly random vertex unequal to $\varnothing$ and $U'$ of degree $\ell$ in $T'$. Then, the following two propositions hold.
\begin{prop}\label{prop:stochdom1}
For any equivalence class $\cC$ under $\sim$, 
\[ \P\big((T,\{U,V\})\in \cC\big)=\P\big((T',\{U',V'\})\in \cC\big).\]
\end{prop}
\begin{prop}\label{prop:stochdom2}
    Let $\cC$ be an equivalence class under $\sim$ with $\p{(T,\{U,V\})\in \cC}>0$. 
    Then, for any $z\in\R$,
    \[\P\big(\H(T)>z \, |\, (T,\{U,V\})\in \cC\big)\le \P\big(\H(T')>z \, |\, (T',\{U',V'\})\in \cC\big).\]
\end{prop}
We first show how Propositions~\ref{prop:stochdom1}~and~\ref{prop:stochdom2} imply Theorem~\ref{thm:stoch_ord}.
\begin{proof}[Proof of Theorem~\ref{thm:stoch_ord} using Propositions~\ref{prop:stochdom1}~and~\ref{prop:stochdom2}]
Fix $r\ge 1$ and let $\nseq=( n_i)_{i\ge 0}$ and $\nseq'=(n'_i)_{i\ge 0}$ be types with $ n_r,n'_r \ge 1$ so that $\nseq'$ covers $\nseq$. Suppose that $\nseq$ has size $n$. Let $T$ be a uniformly random element in $\cT(\nseq)$ and let $T'$ be a uniformly random element in $\cT(\nseq')$ conditional on $\rk_\varnothing(T)=\rk_\varnothing(T')=r$. Then, it is sufficient to show that for any $z\in \R$ it holds that $\p{\H(T)>z}\le \p{\H(T')>z}$.

Recall from Section~\ref{sec:depth_breadth} that $X^{\dfs}(T)$ is the depth-first walk of $T$. If $n'_i= n_i+\I{i=1}$, then a process distributed as $X^\dfs(T')$ can be obtained by sampling  $K$ uniformly on $\{0,\ldots,n\!-\!1\}$ and taking $(0,X^\dfs_1(T),\dots,X^\dfs_{K-1}(T),  X^\dfs_{K}(T), X^\dfs_{K}(T), X^\dfs_{K+1}(T),\dots, X^\dfs_{n}(T))$. The encoded tree (that is distributed as $T'$) is obtained from $T$ by replacing the edge that connects $u^\dfs_{K}(T)$ to its parent by a path of length $2$  when $K\geq 1$ or by adding a parent to the root of $T$ when $K=0$. Thus this modification does not decrease the height of $T$, so it follows that for any $z\in \R$ it holds that $\p{\H(T)>z}\le \p{\H(T')>z}$.  

Secondly, assume that there exist $k, \ell $ with $1\le k\le \ell \le n-1 $ such that 
    \[  n_i = n'_i -\I{i=k} +\I{i=k-1}-\I{i=\ell}+\I{i=\ell+1}.\]
Like before, let $U$ be a uniformly random vertex unequal to $\varnothing$ of degree $k-1$ in $T$, let $V$ be a uniformly random vertex unequal to $\varnothing$ and $U$ of degree $\ell+1$ in $T$, let $U'$ be a uniformly random vertex unequal to $\varnothing$ of degree $k$ in $T'$ and let $V'$ uniformly random vertex unequal to $\varnothing$ and $U'$ of degree $\ell$ in $T'$. Then, writing $\sum_\cC$ for a summation over all $\sim$-equivalence classes $\cC$, we have
\begin{align*}
    \P\big(\H(T)>z\big)&=\sum_\cC\P\big(\H(T)>z \, |\,  (T,\{U,V\}) \in \cC\big)\P\big( (T,\{U,V\}) \in \cC\big)\\
    &=\sum_\cC\P\big(\H(T)>z \, |\, (T,\{U,V\}) \in \cC\big)\P\big( (T',\{U',V'\}) \in \cC\big)\\
    &\le \sum_\cC\P\big(\H(T')>z\, |\, (T',\{U',V'\}) \in \cC\big)\P\big( (T',\{U',V'\}) \in \cC\big)\\
    &=\P\big(\H(T')>z\big),
\end{align*}
where the second equality follows from Proposition~\ref{prop:stochdom1} and the inequality follows from Proposition~\ref{prop:stochdom2}. This proves the statement.
\end{proof}

\begin{proof}[Proof of Proposition~\ref{prop:stochdom1}]
We define a set of trees equipped with two vertices by setting \[E=\{(t,u,v)\in\cT\times \U^2\, :\, u,v\in t\setminus\{\varnothing\}, u\neq v,\rk_v(t)\geq 1\}.\]
With this notation, the laws of $(T,U,V)$ and $(T',U',V')$ are respectively uniform on
\begin{align*}
\{(t,u,v)\in E\, &:\, t\in\cT(\nseq),\rk_{\varnothing}(t)=r, \rk_u(t)=k-1,\rk_v(t)=\ell+1\}, \text{ and }\\
\{(t,u,v)\in E\, &:\, t\in\cT(\nseq'),\rk_\varnothing(t)=r,\rk_u(t)=k,\rk_v(t)=\ell\}.
\end{align*}
The strategy of the proof is to couple $(T,\{U,V\})$ and $(T',\{U',V'\})$ so that they are almost surely equivalent for the relation $\sim$. To do this, we construct a bijective mapping $\Psi:(t,u,v)\in E\longmapsto (t^\star,u^\star,v^\star)\in E$ that satisfies that $(t,\{u,v\})\sim (t^\star,\{u^\star,v^\star\})$ and $\Psi(T,U,V)\eqdist (T',V',U')$. We define $\Psi(t,u,v)$ as follows.

If the last child of $v$ is an ancestor of $u$ or equal to $u$, i.e.~if there is $w\in\U$ such that $u=v*(\rk_v(t))*w$, then we let $t^\star$ be the tree obtained by exchanging the subtrees rooted at the children of $u$ and $v$ respectively (excluding the subtree of $v$ that contains $u$). If we place marks on $u$ and $v$ in $t$, then after exchanging the subtrees, we call the vertices in $t^\star$ with these marks $u^\star$ and $v^\star$ respectively. See Figure~\ref{fig:psi}. Formally, $u^\star=v*(\rk_u(t)+1)*w$ and $v^\star=v$, and $t^\star$ is the tree characterised by the following.
\begin{enumerate}
\item[$(a)$] For all $x\in \U$, $j\in [\rk_u(t)]$, we have that $v*(j)*x\in t^\star$ if and only if $u*(j)*x\in t$;
\item[$(b)$] For all $x\in \U$, $j\in[\rk_v(t)-1]$, we have that $v*(\rk_u(t)+1)*w*(j)*x\in t^\star$ if and only if $v*(j)*x\in t$;
\item[$(c)$] For all $x\in\U$ with $v\not\prec x$, we have $x\in t^\star$ if and only if $x\in t$;
\item[$(d)$] For all $x\in\U$ with $w\not\prec x$, we have that $v*(\rk_u(t)+1)*x\in t^\star$ if and only if $v*(\rk_v(t))*x\in t$;
\item[$(e)$] For all $x\in\U$ and $j\geq 1$, we have that $v*(\rk_u(t)+1+j)*x\notin t^\star$ and that $v*(\rk_u(t)+1)*w*(\rk_v(t)-1+j)*x \notin t^\star$.
\end{enumerate}
Otherwise, i.e.~if $u\neq v*(\rk_v(t))$ and $v*(\rk_v(t))\not\prec u$, then we let $t^\star$ be the tree obtained by removing the subtree rooted at the last child of $v$ and reattaching it after the last child of $u$, and we get $u^\star,v^\star$ by exchanging the roles of $u,v$. See Figure~\ref{fig:psi}. Formally, $u^\star=v$ and $v^\star=u$, and $t^\star$ is the tree characterised by the following.
\begin{enumerate}
\item[$(a)$] For all $x\in \U$, we have $u*(\rk_u(t)+1)*x\in t^\star$ if and only if $v*(\rk_v(t))*x\in t$;
\item[$(b)$] For all $x\in \U$, we have $v*(\rk_v(t))*x \not \in t^\star$;
\item[$(c)$] For all $x\not \in u*(\rk_u(t)+1)*\U\cup v*(\rk_v(t)))*\U$, we have $x\in t^\star$ if and only if $x\in t$. 
\end{enumerate}

\begin{figure}
\centering
    \includegraphics[page=1, scale=0.45]{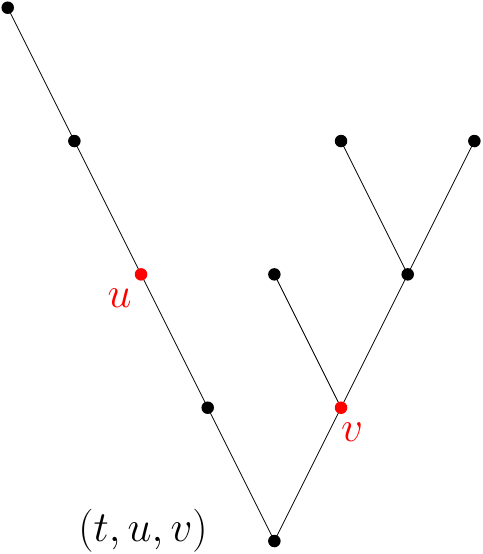}\hfill
    \includegraphics[page=5, scale=0.45]{images/map_psi.pdf}
    \hfill
    \includegraphics[page=3, scale=0.45]{images/map_psi.pdf}\\
    \vspace{1em}
    \includegraphics[page=2, scale=0.45]{images/map_psi.pdf}
    \hfill
      \includegraphics[page=6, scale=0.45]{images/map_psi.pdf}\hfill
    \includegraphics[page=4, scale=0.45]{images/map_psi.pdf} 

    \caption{\label{fig:psi} In the top row, a depiction of three trees with marked points, and in the bottom row their respective images under $\Psi$. In the rightmost example, the last child of $v$ is an ancestor of $u$ so we swap all pendant subtrees; in the other two examples this is not the case so we move one subtree.}
\end{figure}

It is easy to see that in both cases, we indeed have $(t^\star,u^\star,v^\star)\in E$. We also observe that $\rk_\varnothing(t^\star)=\rk_\varnothing(t)$, $\rk_{v^\star}(t^\star)=\rk_u(t)+1$, $\rk_{u^\star}(t^\star)=\rk_v(t)-1$, and that if the type of $t$ is $( n_i(t))_{i\geq 0}$ then the type of $t^\star$ is \[( n_i(t)-\I{i=\rk_u(t)}+\I{i=\rk_u(t)+1}-\I{i=\rk_v(t)}+\I{i=\rk_v(t)-1})_{i\geq 0}.\]
Moreover, by noting that $u^\star\in v^\star*(\rk_{v^\star}(t^\star))*\U$ if and only $u\in v*(\rk_v(t))*\U$, we easily check that $\Psi$ is an involution of $E$, so it is bijective. It follows that if $(T^\star,U^\star,V^\star)=\Psi(T,U,V)$ then $(T^\star,V^\star,U^\star)$ has the same law as $(T',U',V')$.

Finally, we observe that $(t^\star,\{u^\star,v^\star\})\sim (t,\{u,v\})$ by construction. This implies that $(T^\star,\{V^\star,U^\star\})\in \mathcal{C}$ if and only if $(T,\{U,V\})\in\mathcal{C}$, which concludes the proof.
\end{proof}

For the proof of Proposition~\ref{prop:stochdom2}, we use the ``eggs-in-one-basket lemma'', that appears as Lemma 29 in \cite{DonLAB24}. Write ${[n] \choose k}=\{S \subset [n]: |S|=k\}$. Below we use the convention that $\max \emptyset = 0$.
\begin{lem}[Eggs-in-one-basket lemma \cite{DonLAB24}]\label{lem:ei}
Fix two integers $1\le k\le \ell$ and non-negative real numbers $0 < c_1 \le \ldots \le c_{k+\ell}$.
\begin{enumerate}
\item Let $\rA$ be a uniformly random element in ${[k+\ell] \choose k-1} \cup {[k+\ell] \choose \ell+1}$ and let $\rA'$ be a uniformly random element in ${[k+\ell] \choose k} \cup {[k+\ell] \choose \ell}$. Then $\p{\max_{i\in \rA} c_i>z}\le  \p{\max_{i\in \rA'}c_i>z}$ for all $z\in \R$.
\item Let $\rB $ be a uniformly random element in $ {[k+\ell-1] \choose k-1 } \cup {[k+\ell-1] \choose \ell +1 }$ and let $\rB' $ be a uniformly random element in $ {[{k+\ell}-1] \choose k} \cup {[{k+\ell}-1] \choose \ell}$. Then $\p{\max_{i\in\rB}c_i >z}\le \p{ \max_{i\in\rB'}c_i>z}$ for all $z\in \R$. 
\end{enumerate}
\end{lem}

\begin{proof}[Proof of Proposition~\ref{prop:stochdom2}]
We follow the proof of Proposition 28 in \cite{DonLAB24}. Since we assume that $\p{(T,\{U,V\})\in \cC}>0$, we can fix a tree $t$ equipped with two distinct vertices $x,y\in t$ such that $x<y$ and $(t,\{x,y\})\in\mathcal{C}$.

 We first assume that $x$ and $y$ are not ancestrally related. Then, to obtain $T$ conditional on $(T,\{U,V\})\in \mathcal{C}$, we need to sample $(U,V)=(x,y)$ or $(U,V)=(y,x)$ (both with probability $1/2$) and a uniformly random ordering of the $k+\ell$ trees in the multiset $f_{x,y}(t)$, and then define $T$ to be the tree obtained by attaching the first $k-1$ trees in the ordering at $U$ and the last $\ell+1$ trees in the ordering at $V$. Let $t_1,\dots, t_{k+\ell}$ be the trees in the multiset $f_{x,y}(t)$. Thus, for $\rA $ a uniformly random element in $ {[k+\ell] \choose k-1} \cup {[k+\ell] \choose \ell+1}$, the submultiset of trees attached at $y$ has the same law as $\{t_i:i\in \rA\}$. Then, writing $h_0=\H(t_{x,y})$, $h_x=|x|$, $h_y=|y|$, and $c_i=1+\H(t_i)$ for $i\in [k+\ell]$, we see that under the conditioning that $(T,\{U,V\})\in \cC$ we have that 
\[\H(T)\eqdist\max\big(h_0, h_y+\max\{c_i, i\in \rA\}, h_x+\max\{c_i, i \in [k+\ell]\setminus \rA\}\big).\]

Similarly, for $\rA' $ a uniformly random element in $ {[k+\ell] \choose k} \cup {[k+\ell] \choose \ell}$, the submultiset of trees attached at $y$ in $T'$ conditionally on $(T',\{U',V'\})\in \cC$ has the same law as $\{t_i:i\in \rA'\}$, so under this conditioning we have that 
\[\H(T')\eqdist\max\big(h_0, h_y+\max\{c_i, i\in \rA'\}, h_x+\max\{c_i, i \in [k+\ell]\setminus \rA'\}\big).\]

Since $\rA$ and $[k+\ell]\setminus \rA$ have the same distribution, and so do $\rA'$ and $[k+\ell]\setminus \rA'$, we may assume without loss of generality that $h_y\ge h_x$. Thus, under the conditioning, for $M^-=\max(h_0, h_x+\max\{c_i, i \in [k+\ell]\})$ and $M^+=\max(h_0, h_y+\max\{c_i, i \in [k+\ell]\})$, it holds that 
\[M^- \le \H(T)\le M^+\quad\text{ and }\quad M^- \le \H(T')\le M^+.\]
Therefore, we only need to show that for all $M^- \le z< M^+$, it holds that 
\[\p{\H(T)>z \mid (T,\{U,V\})\in \cC}\le \p{\H(T')>z \mid (T',\{U',V'\})\in \cC}.\]
But, we see that for such $z$, 
\begin{align*}
\p{\H(T)>z \mid (T,\{U,V\})\in \cC}&=\p{h_y+\max\{c_i, i\in \rA\}>z}\text{ and }\\\p{\H(T')>z \mid (T',\{U',V'\})\in \cC}&=\p{h_y+\max\{c_i, i\in \rA'\}>z},
\end{align*}
so the first part of the eggs-in-one-basket lemma implies the result.

Now let us assume that $x\prec y$.  The proof proceeds very similar to the previous case. Let $t_1,\dots, t_{k+\ell-1}$ be the trees in the multiset $f_{x,y}(t)$. Then, for $\rB$ a uniformly random element in $\binom{[k+\ell-1]}{k-1} \cup \binom{[k+\ell-1]}{\ell+1}$, the submultiset of trees attached at $y$ in $T$ conditional on $(T,\{U,V\})\in \cC$ has the same law as $\{t_i:i\in \rB\}$. Then, writing $h_0=\H(t_{x,y})$, $h_x=|x|$, $h_y=|y|$, and $c_i=1+\H(t_i)$ for $i\in [k+\ell-1]$, we see that under the conditioning that $(T,\{U,V\})\in \cC$ we have that 
\[\H(T)\eqdist\max\big(h_0, h_y+\max\{c_i, i\in \rB\}, h_x+\max\{c_i, i \in [k+\ell]\setminus \rB\}\big).\]

Similarly, for $\rB'$ a uniformly random element in $\binom{[k+\ell-1]}{k} \cup \binom{[k+\ell-1]}{\ell}$, the submultiset of trees attached at $y$ in $T'$ conditional on $(T',\{U',V'\})\in \cC$ has the same law as $\{t_i:i\in \rB'\}$, so under the conditioning we have that 
\[\H(T')\eqdist\max\big(h_0, h_y+\max\{c_i, i\in \rB'\}, h_x+\max\{c_i, i \in [k+\ell]\setminus \rB'\}\big).\]

Since we assumed that $x\prec y$, we have that $h_y>h_x$. Thus, under the conditioning, for $M^-=\max(h_0, h_x+\max\{c_i, i \in [k+\ell-1]\})$ and $M^+=\max(h_0, h_y+\max\{c_i, i \in [k+\ell-1]\})$, it holds that 
\[M^- \le \H(T)\le M^+\quad\text{ and }\quad M^- \le \H(T')\le M^+.\]
Therefore, we only need to show that for all $M^- \le z< M^+$, it holds that 
\[\p{\H(T)>z \mid (T,\{U,V\})\in \cC}\le \p{\H(T')>z \mid (T',\{U',V'\})\in \cC}.\]
But, we see that for such $z$, 
\begin{align*}
\p{\H(T)>z \mid (T,\{U,V\})\in \cC}&=\p{h_y+\max\{c_i, i\in \rB\}>z}\text{ and }\\
\p{\H(T')>z \mid (T',\{U',V'\})\in \cC}&=\p{h_y+\max\{c_i, i\in \rB'\}>z},
\end{align*}
so the second part of the eggs-in-one-basket lemma implies the result.
\end{proof}

\bibliographystyle{plainnat}
\bibliography{height_width}

\end{document}